\documentclass{article}
\usepackage{
amsmath,
amsfonts,
latexsym, 
amsrefs,
amssymb,
mathtools,
tikz,
tabularx,
minibox,
amsthm,
}
\usepackage{hyperref}
\usepackage{tocloft}
\usepackage[noadjust]{cite}
\usepackage{graphicx}
\usepackage{bm}
\usepackage{dsfont}
\usepackage{stmaryrd}
\usepackage{enumerate}
\usepackage{mathrsfs}
\usepackage{wrapfig}
\usepackage{color}
\usepackage{titlesec}
\usepackage{titletoc}

\usetikzlibrary{shapes}
\usetikzlibrary{intersections}

\numberwithin{equation}{section}
\newcommand{\bla}{\bm{\lambda}}
\newcommand{\bx}{\bm{x}}

\newcommand{\eps}{\varepsilon}

\newcommand{\rd}{{\rm d}}

\newcommand{\bZ}{{\mathbb Z}}
\newcommand{\bu}{{\mathbf u}}
\newcommand{\bv}{{\mathbf v}}

\newcommand{\be}{\begin{equation}}
\newcommand{\ee}{\end{equation}}

\newcommand{\la}{\lambda}

\newcommand{\V}{\mathrm E}
\newcommand{\D}{\mathbb{D}}

\newcommand{\Gin}{\rm{Gin}}
\newcommand{\VGamma}{\Gamma_{V}}
\newcommand{\CUE}{\mathrm{CUE}}

\DeclareMathOperator{\PP}{\mathbb{P}}
\renewcommand{\P}{\mathbb{P}}
\newcommand{\E}{\mathbb{E}}
\newcommand{\R}{\mathbb{R}}
\newcommand{\C}{\mathbb{C}}
\newcommand{\N}{\mathbb{N}}
\newcommand{\Z}{\mathbb{Z}}

\newcommand{\dd}{\mathrm{d}}

\theoremstyle{plain} 
\newtheorem{theorem}{Theorem}[section]
\newtheorem*{theorem*}{Theorem}
\newtheorem{lemma}[theorem]{Lemma}
\newtheorem*{lemma*}{Lemma}
\newtheorem{corollary}[theorem]{Corollary}
\newtheorem*{corollary*}{Corollary}
\newtheorem{proposition}[theorem]{Proposition}
\newtheorem*{proposition*}{Proposition}
\newtheorem{fact}[theorem]{Fact}
\newtheorem*{fact*}{Fact}

\newtheorem{definition}[theorem]{Definition}
\newtheorem*{definition*}{Definition}

\newtheorem*{example*}{Example}
\newtheorem{remark}[theorem]{Remark}

\newtheorem*{remark*}{Remark}
\newtheorem*{remarks*}{Remarks}

\makeatletter
\renewcommand{\section}{\@startsection
{section}
{1}
{0mm}
{-2\baselineskip}
{1\baselineskip}
{\normalfont\large\scshape\centering}} 
\makeatother

\makeatletter
\renewcommand{\subsection}{\@startsection
{subsection}
{2}
{0mm}
{-\baselineskip}
{1 \baselineskip}
{\normalfont\scshape}} 
\makeatother

\makeatletter
\renewcommand{\subsubsection}{\@startsection{subsubsection}{3}{\z@}%
  {3.25ex \@plus 1ex \@minus .2ex}{-1em}{\normalfont\normalsize\itshape}}
\makeatother
\usepackage[T1]{fontenc}

\renewcommand{\rm}{\mathrm}

\def\@empty{}

\def\author#1{\par
    {\centering{\authorfont#1}\par\vspace*{0.05in}}}

\def\titlefont{\fontsize{12}{15} \centering{}}
\def\authorfont{\fontsize{12}{15}}

\let\affiliationfont\rhfont

\def\address#1{\par
    {\centering{\affiliationfont#1\par}}\par\vspace*{11pt}
}

\def\keywords#1{\par
    \vspace*{8pt}
    {\authorfont{\leftskip18pt\rightskip\leftskip
    \noindent{\it\small{Keywords}}\/:\ #1\par}}\vskip-12pt}

\def\body{
\setcounter{footnote}{0}
\def\thefootnote{\alph{footnote}}
\def\@makefnmark{{$^{\rm \@thefnmark}$}}
}

\def\title#1{ 
    \thispagestyle{plain}
    \vspace*{-14pt}
    \vskip 79pt
    {\centering{\titlefont #1\par}}%
    \vskip 1em
}

\begin{document}

~\vspace{-2cm}

\title{\textsc{Powers of Ginibre Eigenvalues}}

\vspace{0.3cm}

 \author{Guillaume Dubach}

\address{Courant Institute, New York University \\
dubach@cims.nyu.edu}
\vspace{0.3cm}

\begin{abstract}
\noindent We study the images of the complex Ginibre eigenvalues under the power maps $\pi_M: z \mapsto z^M$, for any integer $M$. We establish the following equality in distribution,
$$ 
\rm{Gin}(N)^M \stackrel{d}{=} \bigcup_{k=1}^M \rm{Gin} (N,M,k),
$$
where the so-called Power-Ginibre distributions $\rm{Gin}(N,M,k)$ form $M$ independent determinantal point processes. The decomposition can be extended to any radially symmetric normal matrix ensemble, and generalizes Rains' superposition theorem for the CUE (see \cite{Rains}) and Kostlan's independence of radii (see \cite{Kost}) to a wider class of point processes. Our proof technique also allows us to recover and generalize a result by Edelman and La Croix \cite{EdelmanLacroix} for the GUE. \medskip

Concerning the Power-Ginibre blocks, we prove convergence of fluctuations of their smooth linear statistics to independent gaussian variables, coherent with the link between the complex Ginibre Ensemble and the Gaussian Free Field \cite{RiderVirag}.  \medskip

Finally, some partial results about two-dimensional beta ensembles with radial symmetry and even parameter $\beta$ are discussed, replacing independence by conditional independence.
\end{abstract}

\vspace{.5cm}

\tableofcontents

\newpage
\section{Introduction}
\subsection{Motivations}
\noindent The complex Ginibre ensemble, that we will denote by $\Gin(N)$, consists of matrices whose coefficients are independent and identically distributed complex Gaussian random variables. With the appropriate scaling,
\begin{equation}\label{eqn:Gini}
G =(G_{ij})_{i,j=1}^N,
\qquad
G_{ij} \stackrel{d}{=} \mathscr{N}_{\mathbb{C}} \Big(0, \frac{1}{N} \Big).
\end{equation}

It has been known since the seminal work of Ginibre (see \cite{Ginibre}) that the eigenvalues of such a matrix form a determinantal point process. The joint density is proportional to
$$\prod_{1 \leq i<j \leq N} |z_i-z_j|^2 e^{- N \sum_{i=1}^N |z_i|^2}$$
with respect to the Lebesgue measure on $\mathbb{C}$. \medskip

This density shows strong interaction (repulsion) between eigenvalues. However, remarkably, Kostlan has shown \cite{Kost} that their {\it radii} are independent in the following sense.


\begin{theorem}[Kostlan]\label{Kostlan?} If $\{ \lambda_1, \dots, \lambda_N \}$ is distributed according to $\Gin(N)$, he following equality in distribution holds :
$$
\{ N |\lambda_1|^2, \dots, N |\lambda_N|^2 \}
\stackrel{d}{=}
\{ \gamma_1, \dots, \gamma_N \},
$$
where the gamma variables are independent, with parameters $1,2, \dots, N$.
\end{theorem}

The same holds in the more general setting of radially symmetric point processes. With techniques reviewed in \cite{HKPV}, Hough, Krishnapur, Peres and Vir\'ag established a broader version of Kostlan's theorem, as well as the independence of high powers. In the Ginibre case, it can be stated as follows.

\begin{theorem}[Hough, Krishnapur, Peres, Vir\'ag]\label{HighPowers1} For any integer $M \geq N$, the following equality in distribution holds:
$$
\{ N^{M/2} \lambda_1^M, \dots, N^{M/2} \lambda_N^M \} 
\stackrel{d}{=}
\{\gamma_1^{M/2} e^{i \theta_1}, \dots, \gamma_N^{M/2} e^{i \theta_N} \}
$$
where the variables $\gamma_k,\theta_k$ are independent, the gamma variables having parameters $1,2, \dots, N$, and the angles being uniform on $[0,2\pi]$.
\end{theorem}

Note that this is not an asymptotic, but an exact result.


In specific settings, two other results hinted that something unusual happens with quadratic repulsion that would concern more than the radii or the high powers only. The first of these results was stated for the eigenvalues of a Haar-distributed unitary matrix (known as the Circular Unitary Ensemble, or CUE). The joint density of these eigenvalues is proportional to
$$\prod_{1 \leq k<j \leq N} |e^{i \theta_k}-e^{i \theta_j}|^2 $$
with respect to the Lebesgue measure on the unit circle. Rains established in~\cite{Rains} a decomposition in $M$ independent CUE blocks, for any power of the eigenvalues of $\CUE(N)$. 

\begin{theorem}[Rains]
For any $M \geq 1$,
$$
\CUE(N)^M \stackrel{d}{=} \bigcup_{k=1}^N \CUE \left(\left\lceil \frac{N -k}{M} \right\rceil \right).
$$
where the Circular Unitary Ensembles in the right hand side are independent.
\end{theorem}

The second example involves the eigenvalues of Gaussian hermitian matrices (the Gaussian Unitary Ensemble, or GUE). This example is different in this,that it does not exhibit radial symmetry on $\C$. 
GUE eigenvalues are distributed on $\R$ with joint density proportional to :
$$\prod_{1 \leq i<j \leq N} |x_i-x_j|^2 e^{- N \sum x_i^2}.$$

Edelman and La Croix \cite{EdelmanLacroix} established a block decomposition, which holds for the squares of the GUE eigenvalues (that is, the singular values of a GUE matrix) and is made of two independent Laguerre blocks (see \cite{EdelmanLacroix} for a definition).
 
\begin{theorem} [Edelman-Lacroix]
$$
\mathrm{GUE}(N)^2 \stackrel{d}{=} \mathrm{LUE}(N,1) \cup \mathrm{LUE}(N,2),
$$
where $\rm{LUE}(N,k)$ stand for Laguerre ensembles with half-integer parameters. 
\end{theorem}
The fact that such a result holds only for the squares is essentially due to the lack of symmetry.

\subsection{Results}

Our results are essentially a generalization of the above. Instead of relying on the underlying matrix ensemble, or actually decomposing the density, we rely on a characterization of the law by the statistics obtained with the so-called product symmetric polynomials, that is, expressions of the type
$$
\E \left( \prod_{i=1}^N P(\lambda_i, \overline{\lambda_i}) \right)
$$
for any polynomial $P$ in two variables. Such statistics can be exactly computed thanks to Andr\'eief's identity, and they characterize the point process (see the Appendix). We illustrate this in Subsection \ref{ProofTechnique} by giving a new proof of Kostlan's theorem and the independence of $M$th powers for $M \geq N$. \medskip

Our main new result follows this approach. This is the identity stated in the abstract, that we refer to as the Power-Ginibre decomposition, summarized in Figure 1.

\begin{theorem}[Power-Ginibre Decomposition]\label{PG} For fixed integers $M \leq N$, let us define the sets
$$
I_k = \{ i \in \llbracket 1,N \rrbracket \ | \ i \equiv k \ [M] \}, \qquad 1 \leq k \leq M. 
$$
The following equality in distribution holds, when $\{\lambda_1, \dots, \lambda_N\}$ is distributed according to $\Gin(N)$,
$$
\{\lambda_1^M, \dots, \lambda_N^M\} \stackrel{d}{=} \bigcup_{k=1}^M \Gin(N,M,k)
$$
where the independent Power-Ginibre distributions $\Gin(N,M,k)$ are indexed by the sets $I_k$, with joint densities
$$
\frac{1}{Z_{N,M,k}}
\prod_{\substack{i<j \\ i,j \in I_k} } |z_i - z_j|^2  \prod_{i \in I_{k}} |z_i|^{\frac{2(k-M)}{M}}  e^{- N |z_i|^{2/M} } \dd m (z_i).
$$
\end{theorem}

\begin{figure}\label{fig:spirals}
\centering
\begin{tikzpicture}

 \fill [white,fill opacity=1] (6,0) -- plot [domain=0:pi,samples=100] ({6+1.7*sin(\x r)},{0.5*cos(\x r)-0.1*\x-1-0.4*pi}) -- (6,0) -- cycle;

  \draw [name path=spirale] [black,thick,-] plot [domain=0:pi,samples=100] ({6+1.7*sin(\x r)},{0.5*cos(\x r)-0.1*\x-1-0.4*pi});
 
  \fill [white,fill opacity=1] (6,0) -- plot [domain=0:2*pi,samples=100] ({6+1.7*sin(\x r)},{0.5*cos(\x r)-0.1*\x-1-0.2*pi}) -- (6,0) -- cycle;
 
 \draw [name path=spirale] [black,thick,-] plot [domain=0:2*pi,samples=100] ({6+1.7*sin(\x r)},{0.5*cos(\x r)-0.1*\x-1-0.2*pi});

 \fill [white,fill opacity=1] (6,0) -- plot [domain=0:2*pi,samples=100] ({6+1.7*sin(\x r)},{0.5*cos(\x r)-0.1*\x-1}) -- (6,0) -- cycle;

\draw [name path=spirale] [black,thick,-] plot [domain=0:2*pi,samples=100] ({6+1.7*sin(\x r)},{0.5*cos(\x r)-0.1*\x-1});

\draw [-] (6,-0.5) -- (6,-1.4);
\draw [-] (6,-1.79) -- (6,-2.1);
\draw [-] (6,-2.45) -- (6,-3.1);

\draw (6.7,-1.2) node[scale=0.9,rotate=0]{$\Gin(N)$};

\draw [->] (6,-3.5) -- (6,-4.1);

 \draw (7,-3.7) node[scale=0.9,rotate=0]{$\pi_M : z \mapsto z^M$};
\node[draw,ellipse] (S4)at(6,-4.6) {$ \quad \Gin(N)^M \quad $};
\end{tikzpicture}
\hspace{1cm}
\begin{tikzpicture}
\node[draw,ellipse] (S2)at(6,-1.6) {$ \ \Gin(N,M,2) \ $};
\node[draw,ellipse,fill,white] (S1)at(6,-0.9) {$ \Gin(N,M,1) $}; 
\node[draw,ellipse] (S1)at(6,-0.9) {$ \ \Gin(N,M,1) \ $};
\node[circle,fill,inner sep=0.3pt](d)at(6,-2.1){};
\node[circle,fill,inner sep=0.3pt](d)at(6,-2.2){};
\node[circle,fill,inner sep=0.3pt](d)at(6,-2.3){};
\node[draw,ellipse] (S3)at(6,-2.8) {$\Gin(N,M,M)$};
\draw [->] (6,-3.5) -- (6,-4.1);
\draw (7.2,-3.7) node[scale=0.9,rotate=0]{(Superposition)};
\node[draw,ellipse] (S4)at(6,-4.6) {$ \quad \Gin(N)^M \quad $};
\end{tikzpicture}

\caption{
In the first picture, every pair of eigenvalues experiences repulsion, but the images of two eigenvalues under $\pi_M$ may still be close because they come from different sheets of the covering map $\pi_M$. In the second picture, every sheet has quadratic repulsion within itself, but each sheet is independent from the others.}
\end{figure}
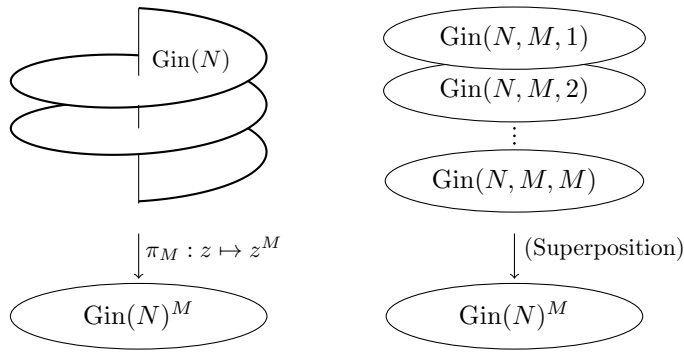

In Subsection \ref{Assymp}, we study the linear statistics of the Power-Ginibre distributions, using the determinantal structure it exhibits. The first order given by Theorem \ref{twist} involves the pushforward of the circular law by $\pi_M$, coherently with a classical result. The second order given by Theorem \ref{GFF2} is in accordance with the Gaussian Free Field limit established by Rider and Vir\'ag in \cite{RiderVirag}. \medskip

Section \ref{Generalization} goes beyond the Ginibre case to study ensembles with different supports and features. Indeed, our proof technique argument does not depend on the reference measure, but only relies on quadratic repulsion, provided that the distributions involved exhibit radial symmetry, as stated in \ref{PVDec}. Such distributions include, for instance, products of independent Ginibre matrices, minors of the CUE, or the spherical ensemble. These examples are given in more details in Subsection \ref{PVDec}. \medskip

Our general Power Decomposition encompasses Rains' result for the CUE, as we state precisely in Subsection \ref{CUE}. A peculiarity of the CUE case is the fact that its characteristic polynomial on the circle is distributed as a product of independent terms, a property first proved in \cite{BHNY}. We give a new proof of this fact relying on the techniques we introduce in this paper. In this case, however, our approach does not allow us to extend the result any further. \medskip

Subsection \ref{GUE} focuses on the GUE. In this case, block decomposition does not hold for all powers, but only for $M=2$. This generalizes the result of Edelman and La Croix. We are also able to provide a direct proof of another fact they mention: as in the $\CUE$ case, the characteristic polynomial at a specific point ($z=0$ here) is distributed like a product of independent variables. \medskip

Subsection \ref{BetaDec} initiates a further generalization of our results to two-dimensional beta ensembles with radial symmetry, when $\beta$ is an even integer. In that case, independence needs be replaced by a form of conditional independence with an explicit discrete variable $I$ that appears as a random environment. We generalize our study of squared radii and high powers to this context. However, the lack of a suitable form of Andr\'eief's identity prevents us from studying intermediate powers. \medskip

\newpage
\subsection{Synoptic table} For the convenience of the reader, we provide a table summing up all results considered in our paper about three ensembles of random matrix theory : namely, the complex Ginibre ensemble, the Circular Unitary Ensemble, and the Gaussian Unitary Ensemble. The three relevant aspects mentioned are the behavior of the radii, the existence of a decomposition of all or some powers, and the distribution of the characteristic polynomial at specific points. 

\vspace{0.5cm}

{\small
\begin{tabular}{|l|c|c|c|}  
  \hline
   &  &  & \\
   & Squared Radii  & Powers &  \minibox{Characteristic \\ \ \ Polynomial } 
    \\
    &  &  &  \\
  \hline
 &  &  & \\
  & Independence & Block Independence & Independence at $z=0$ \\
 Ginibre & $\{ N |\lambda_k|^2 \}_{k=1}^N \stackrel{d}{=} \{ \gamma_k \} _{k=1}^N$ & $ \rm{Gin}(N)^M \stackrel{d}{=} \displaystyle{\bigcup_{k=1}^M} \rm{Gin} (N,M,k) $ & $ P_N(0) \stackrel{d}{=} e^{i \theta} \displaystyle{\prod_{k=1}^N} \sqrt{\gamma_k }$ \\
  & (Kostlan \cite{Kost}) & (P.-G. Decomposition) & \\
   &  &  & \\
 \hline
  &  &  & \\
  & (Trivial) Independence & Block Independence & Independence at $z=1$ \\
 $\CUE$ & $ |e^{i \theta}|^2=1$ & $
\CUE(N)^M \stackrel{d}{=} \displaystyle{\bigcup_{k=1}^N} \CUE \left(\left\lceil \frac{N -k}{M} \right\rceil \right)
$ &  $Z_N(1) \stackrel{d}{=}  \displaystyle{\prod_{k=1}^N} \left(1 + \sqrt{\beta_{1,k-1}} e^{i \theta_k}\right)$ \\
  &  & (Rains \cite{Rains})  &(Bourgade \& al. \cite{BHNY}) \\
   &  &  & \\
 \hline
  &  &  & \\
 & Block Independence & Block Independence for $M=2$  & Independence at $z=0$ \\
  $\mathrm{GUE}$& (Same as powers, for $M=2$) &  $\mathrm{GUE}(N)^2 \stackrel{d}{=} \displaystyle{\bigcup_{k=1}^2} \mathrm{LUE}(N,k) $ &  $|P_N(0)| \stackrel{d}{=} \displaystyle{\prod_{k=1}^N} \chi^2\left({2\left\lfloor \frac{k}{2} \right\rfloor +1}\right)$ \\
  &  & (Edelman-La Croix \cite{EdelmanLacroix}) & (Edelman-La Croix \cite{EdelmanLacroix}) \\
   &  &  & \\
  \hline
\end{tabular}
} \medskip

Our new result here is the Power-Ginibre decomposition, which fills the last gap in this global picture. The results about the characteristic polynomial of the CUE and GUE are not new, but we recover them by another method.

\subsection{Notations and conventions}

\noindent We denote $\dd m$  the Lebesgue measure on $\mathbb{C}$, $\dd m^N(\mathbf{z})$ the Lebesgue measure on $\mathbb{C}^N$, and the standard complex Gaussian distributions by
$$\dd \mu( z) = \frac{1}{\pi} e^{- |z|^2} \dd m(z),
\qquad
\dd \mu^{(N)}( \mathbf{z}) = \frac{N^N}{\pi^N} e^{- N \| \mathbf{z} \|^2} \dd m^N( \mathbf{z}).$$ 
Note that $\dd \mu^{(N)}$ differs from $\dd \mu^{\otimes N}$ in scaling. \medskip

The complex Ginibre ensemble $\Gin(N)$ is defined by (\ref{eqn:Gini}), with the appropriate $N^{-1/2}$ scaling. We will sometimes refer to the matrix $\sqrt{N} G$ and its eigenvalues as the {\it unscaled} Ginibre ensemble. \medskip

The joint density of the eigenvalues of $\Gin(N)$ is given by
\begin{equation}
\frac{1}{Z_N} \prod_{i<j} |z_i - z_j|^2 \dd \mu^{(N)}(z_1, \dots, z_N),
\end{equation}
where $Z_N=N^{-N(N-1)/2}\prod_{j=1}^N j!$, and it is known (see \cite{Ginibre}) that the limiting empirical spectral measure converges weakly to the circular law,
$$
\frac{1}{N}\sum \delta_{\lambda_i} \stackrel{d}{\rightarrow} \frac{1}{\pi}\mathds{1}_{ \{ |\lambda|<1 \} }\rd m(\lambda).
$$


The power map $\pi_M: z \mapsto z^M$ is a covering map of $\mathbb{C}^*$ with $M$ sheets. It conformally maps the slice $\angle_M := \{ 0 < \arg(z)< \frac{2\pi}{M} ,0< |z|<1 \}$ to $\D \backslash [0,1]$, and the simple change of variables $\omega=\pi_M(z)$ gives \begin{equation}\label{CoV}
\frac{1}{\pi} \int_{\D} g(z^M) \dd m(z) = \frac{M}{\pi} \int_{\angle_M} g(z^M) \dd m(z)
= \frac{1}{\pi M } \int_{\D} g(\omega) |\omega|^{2/M-2} \dd m(\omega),
\end{equation}
where the weight $\frac{1}{M} |\omega|^{2/M-2}$ corresponds to the concentration of the measure at the origin displayed on Figure~1. We will refer to the associated measure on the unit disk as the twisted (or $M$-twisted) circular distribution; it is the pushforward of the circular distribution by $\pi_M$.\medskip
\begin{figure}[!b]
\begin{center}
\includegraphics[width=\textwidth]{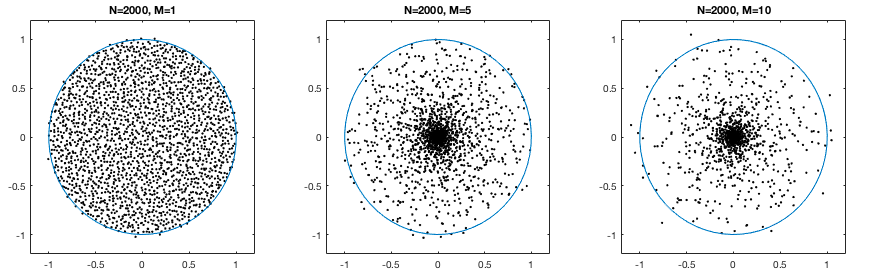}
\caption{MATLAB simulations of low powers of $\Gin(2000)$.}
\end{center}
\end{figure}

In Section \ref{PGDec} we will consider the partial sums of the exponential along an arithmetic progression. For any integers $ M \geq k \geq 1$ and $n$, we will denote
$$
e_{M,k}^{(n)}(x) = \sum_{j=0}^{n-1} \frac{x^{Mj+k-1}}{(Mj+k-1)!}, \qquad  e_{M,k} (x) = \sum_{j \geq 0} \frac{x^{Mj+k-1}}{(Mj+k-1)!}.
$$
It is clear that, if $\zeta$ is a primitive $M$th root of unity, for any $l \in [\![ 1,M ]\!]$,
$$
e^{\zeta^{l-1} x} = \sum_{k=1}^M \zeta^{(l-1)(k-1)} e_{M,k} (x),
$$
and one can reverse this identity through the Vandermonde matrix $\left(\zeta^{(i-1)(j-1)}\right)_{i,j=1}^M$, so that it reads
\begin{equation}\label{partialsums}
e_{M,k} (x) = \frac{1}{M} \sum_{l=1}^M {\overline{\zeta}}^{(k-1)(l-1)} e^{(\zeta^{l-1} x)}.
\end{equation}
Such identities will be used in Section \ref{PGDec}. \medskip

Unless otherwise specified, the capital letters $X,Y$ stand for random variables, whereas $T,S$ stand for polynomial indeterminates. Capital $Z$ can denote one or the other, depending on the context.

\section{The Complex Ginibre Ensemble under power maps}

The technique that will be used to prove the main theorem is first exemplified in Subsection \ref{ProofTechnique} in order to recover a few well-known results. We then proceed to prove Theorem \ref{PG2}, showing that every power map gives rise to some decomposition in independent blocks. Subsection \ref{Assymp} is devoted to a few asymptotic properties of these Power-Ginibre blocks, thus checking that their existence is coherent with two main features of the complex Ginibre ensemble : namely, the circular law, and the Gaussian Free Field.

\subsection{Distribution of radii and large powers}\label{ProofTechnique}
In this section we provide a new proof of some known results. The technique we use is essentially the same as the one that gives rise to new results in Sections \ref{PGDec} and \ref{PVDec}: characterization of the distribution of a set of complex variable through product statistics obtained with Andr\'eief's identity. One technical issue is to prove that such statistics indeed characterize the distribution. This classical technicality is dealt with in the Appendix.

\subsubsection{Andr\'eief's identity and product statistics.}
We following lemma is a key-fact in the study of determinantal processes.
\begin{lemma}[Andr\'eief]\label{Andr} Let $(E, \mathcal{E}, \nu)$ be a measure space. For any functions $(\phi_i, \psi_i)_{i=1}^N \in L^2(\nu)^{2N}$,
$$
\frac{1}{N!} \int_{E^N} \det \left(  \phi_i(\lambda_j) \right) \ \det  \left( \psi_i (\lambda_j) \right) \ \dd \nu^{\otimes N} ({\bf \lambda}) = \det \left(f_{i,j}\right), \quad \text{where} \ f_{i,j} = \int_E \phi_i(\lambda) \psi_j(\lambda) \dd \nu( \lambda).
$$
 \end{lemma}

\noindent A proof can be found in \cite{Andreiev,Deift1}. For Ginibre, it yields the following explicit formula for the product statistics.

\begin{corollary}[Product Statistics]\label{ProStat} Let $E=\mathbb{C}$, $g \in L^2(\mu)$, and  $\{ \lambda_1, \dots, \lambda_N\}$ Ginibre eigenvalues. Then,
$$ \E \left( \prod_{k=1}^N g \left(\lambda_k \sqrt{N} \right) \right)= \det \left( f_{i,j} \right)_{i,j = 1}^{N} \quad \text{where} \ f_{i,j} = \frac{1}{(j-1)!} \int z^{i-1} \bar{z}^{j-1} g(z) \dd \mu ( z).$$ 
\end{corollary}

\begin{proof} The unscaled eigenvalues $\{ \lambda_1 \sqrt{N}, \dots, \lambda_N \sqrt{N}\}$ have joint density
$$
\frac{1}{\prod_{j=1}^N j!} \left| \det \left( z_i^{j-1} \right)_{i,j=1}^N \right|^2 \dd \mu^{\otimes N}(z_1, \dots, z_N).
$$
Therefore, using Lemma \ref{Andr} with $\phi_i(z) = z^{i-1} g(z)$, $\psi_j(z)= z^{j-1}$ and $\dd \nu = \dd \mu$ yields
$$
\E \left( \prod_{k=1}^N g \left(\lambda_k \sqrt{N} \right)  \right)= \frac{N!}{\prod_{k=1}^N k!} \det \left( \int z^{i-1} \bar{z}^{j-1} g(z) \dd \mu ( z) \right)_{i,j=1}^N = \det \left(f_{i,j} \right).
$$
as was claimed.
\end{proof}

\subsubsection{Kostlan's theorem.}
\noindent Corollary \ref{ProStat} implies Kostlan's theorem, provided the statistics of a set of real random variables are fully characterized by these product statistics. We define the product symmetric polynomials as the symmetric polynomials given by products of polynomials in one variable,
$$ \rm{PS}_{\C}(N) = \left\{ \ \prod_{i=1}^N P(T_i) \ | \ P \in \mathbb{C} [T] \ \right\}. $$ 
\begin{lemma}\label{sympol}
$\rm{PS}_{\C}(N)$ spans $\rm{S}_{\C}(N)$ as a vector space.
\end{lemma}

The proof can be found in the Appendix. We now give a proof of Kostlan's theorem, that we restate for the convenience of the reader :

\begin{theorem}[Kostlan]\label{Kostlan}
$
\{ N |\lambda_1|^2, \dots, N |\lambda_N|^2 \}
\stackrel{d}{=}
\{ \gamma_1, \dots, \gamma_N \},
$
where the gamma variables are independent, with parameters $1,2, \dots, N$.
\end{theorem}

\begin{proof} Let $g \in \mathbb{C}[X]$ and apply Corollary \ref{ProStat} to the radially symmetric function $g(|\cdot|^2)$. The matrix is then diagonal, with coefficients
$$
f_{i,i} = \frac{1}{(i-1)!} \int |z|^{2i-2} g(|z|^2) \dd \mu ( z) = \frac{1}{(i-1)!} \int_{r=0}^{\infty} r^{i-1} g(r) e^{-r} \dd r = \E \left( g(\gamma_i) \right).
$$
That is to say,
$$
\E \left( \prod_{i=1}^N g(N |\la_i|^2) \right) =  \E \left(  \prod_{i=1}^N g(\gamma_{i}) \right). 
$$
These statistics characterize the distribution of a set of points, as such expressions with polynomial $g$ generate all symmetric polynomials (see Lemma \ref{sympol}), and the distributions involved are characterized by their moments. We conclude that $ \{ N |\lambda_1|^2, \dots, N |\lambda_N|^2 \} \stackrel{d}{=} \{ \gamma_1, \dots, \gamma_N\}$.
\end{proof}

\subsubsection{Independence of high powers.}

\noindent An equivalent characterization of distributions holds for sets of random complex variables, with mixed symmetric polynomial in $Z,\overline{Z}$. In this section, $Z$ stands for a polynomial indeterminate and not a specific random variable. We write $Z=T+iS$ and $\overline{Z}=T-iS$. The product symmetric mixed polynomials are the symmetric polynomials given by products of mixed polynomials, that is, polynomials in one variable and its complex conjugate:
$$
\rm{PMS}_{\C}(N) = \left\{ \ \prod_{j=1}^N P(Z_j, \overline{Z}_j) \ | \ P \in \mathbb{C} [T,S] \ \right\}
$$ 
Even though the notation $P(Z, \overline{Z})$ is redundant, we use it to make it clear that we are dealing with a mixed polynomial. We now extend Lemma \ref{sympol} to $\rm{PMS}_{\C}(N)$. This will allow us to characterize the distribution of a set of complex variables by examining its product statistics.

\begin{lemma}\label{sympol2}
$\rm{PMS}_{\C}(N)$ spans $\rm{MS}_{\C}(N)$ as a vector space.
\end{lemma}

The proof can be found in the Appendix. We now prove the following result, that was announced in the introduction, restated here for the convenience of the reader.

\begin{theorem}[Hough, Krishnapur, Peres, Vir\'ag]\label{HighPowers2} For any integer $M \geq N$, the following equality in distribution holds:
$$
\{ N^{M/2} \lambda_1^M, \dots, N^{M/2} \lambda_N^M \} 
\stackrel{d}{=}
\{\gamma_1^{M/2} e^{i \theta_1}, \dots, \gamma_N^{M/2} e^{i \theta_N} \}
$$
where the variables $\gamma_k,\theta_k$ are independent, the gamma variables having parameters $1,2, \dots, N$, and the angles being uniform on $[0,2\pi]$.
\end{theorem}

\begin{proof} Let $g \in \mathbb{C}[X,\overline{X}]$ and apply Corollary \ref{ProStat} to the polynomial $g(X^M)$. If we call relative degree of a monomial the difference between its degrees in the first and second variable,
$$\mathrm{reldeg} = \deg_X - \deg_{\overline{X}},$$
the monomials of relative degree $0$ are the powers of $|X|^2$, and the relative degree of a product is the sum of the relative degrees. Lemma \ref{ProStat} gives
$$ \E \left( \prod_{k=1}^N g ( \lambda_k^M N^{M/2} )  \right)= \det \left( f_{i,j} \right)_{i,j = 1}^{N} \quad \text{where} \ f_{i,j} = \frac{1}{(j-1)!} \int z^{i-1} \bar{z}^{j-1} g(z^M) \dd \mu ( z).$$ 
The relative degrees of the monomials of $g(X^M)$ are multiples of $M \geq N$, but on the other hand
$$|\mathrm{reldeg}(X^{i-1} \overline{X}^{j-1})| = | i-j | <N $$
Expanding the expression $z^{i-1} \bar{z}^{j-1} g(z^M)$ as a sum of monomials, it is clear that only the monomials with relative degree $0$ contribute, and these can be achieved only for $i=j$. The matrix is therefore diagonal, with
\begin{align*}
f_{j,j} & = \frac{1}{(j-1)!} \int |z|^{2j-2} g(z^M) \dd \mu ( z) \\
&= \frac{1}{2 \pi (j-1)!} \int_{\theta=0}^{2 \pi} \int_{r=0}^{\infty} r^{j-1} g(r^{M/2} e^{i \theta_j}) e^{-r} \dd r \dd \theta_j \\
& = \E \left( g(\gamma_j^{M/2} e^{i \theta_j})) \right).
\end{align*}
That is to say,
$$
\E \left( \prod_{j=1}^N g( \la_j^M N^{M/2}) \right) = \prod_{j=1}^N \E \left(  g(\gamma_{j}^{M/2} e^{i \theta_j}) \right) =  \E \left(  \prod_{j=1}^N g(\gamma_{j}^{M/2} e^{i \theta_j}) \right). 
$$
These statistics characterize the distribution of a set of points, as such expressions with polynomial $g$ generate all symmetric polynomials (see Lemma \ref{sympol2}), and the distributions involved are characterized by their moments. The result follows.
\end{proof}

\subsection{Power-Ginibre decomposition}\label{PGDec}

We state and prove here our main result, Theorem \ref{PG} : decomposition of the images of the complex Ginibre Ensemble under a power map as independent blocks, this for any power $M$. This relies on the techniques introduced above, and requires first a few elementary definitions.

\subsubsection{Arithmetic progressions and determinants of striped matrices.}
We consider the total number of Ginibre points $N$, an integer $M \in \N$, and write
$$N=qM+r, \qquad 0 \leq r < M$$
the Euclidean division of $N$ by $M$. By arithmetic progressions we mean the intersections of $\llbracket 1,N \rrbracket$ with infinite arithmetic progressions of step $M$.

\begin{remark}
The set $\llbracket 1,N \rrbracket$ is partitioned by $r$ arithmetic progressions of length $q+1$ and $M-r$ arithmetic progressions of length $q$. These are given by the sets
\begin{equation}\label{Progression}
I_{N,M,k}= \{ i \in \llbracket 1,N \rrbracket \ | \ i \equiv k \ [M] \}, \qquad 1 \leq k \leq M,
\end{equation}
whose cardinalities depend on whether $k \leq r$ or $k>r$.
\end{remark}
\noindent We will sometimes use the notation
$$ I_k=I_{N,M,k}, \quad c_k= |I_k|.$$

\begin{definition}
We say that a matrix $A \in \mathcal{M}_N(\mathbb{C})$ is $M$-striped if
$$ i-j \not \equiv 0 \ [M] \ \Rightarrow \ A_{i,j}=0.$$
\end{definition}

The determinant of an $M$-striped matrix can be factorized as the product of $M$ determinants, as shown below.

\begin{lemma}\label{Mstriped} The determinant of an $M$-striped matrix $A \in \mathcal{M}_N(\mathbb{C})$ is the product of its minors indexed by the arithmetic progressions $I_{N,M,k}$. That is, if $A_k=(A_{i,j})_{i,j \in I_k}$, then
$$
\det(A)= \prod_{k=1}^M \det(A_k).
$$
\end{lemma}

\begin{proof} The $M$-striped matrix $A$ is equivalent to a block matrix, by conjugation with a permutation matrix that re-indexes $\llbracket1,N\rrbracket$ according to  $I_1, I_2, \dots, I_M$. \end{proof}

\subsubsection{The Power-Ginibre distributions.}
\noindent Even though the relevant determinantal process is the one we define below as $\Gin(N,M,k)$, for convenience in the proofs, we also define its preimage by the power map, $\rm{R}(N,M,k)$.
\begin{definition} For any triple $(N,M,k)$ 
we define the root distribution $\rm{R}(N,M,k)$ as the point process indexed by $I_{N,M,k}$ with joint density
$$
\frac{1}{Z_{\rm{R}(N,M,k)}}  \prod_{\substack{i<j \\ i,j \in I_k}} N^M |z_i^M - z_j^M|^2 \prod_{i \in I_{k}} N^{k} |z_i|^{2(k-1)} \dd \mu (z_i),
$$
where
$$
Z_{\rm{R}(N,M,k)} = c_k !  \prod_{j \in I_{k}} (j-1)!.
$$
The {Power-Ginibre} distribution $\Gin(N,M,k)$ is the image of the root distribution under the power map $\pi_M$. Equivalently, $\Gin(N,M,k)$ is the point process indexed by $I_{k}$ with joint density
$$ \frac{1}{Z_{N,M,k}}
\prod_{\substack{i<j \\ i,j \in I_k}} N^M |z_i - z_j|^2  \prod_{i \in I_{k}} N^{k} |z_i|^{\frac{2(k-M)}{M}}  e^{- N |z_i|^{2/M} } \dd m (z_i),
$$
where
$$
Z_{N,M,k} = \pi^{c_k} c_k! \ {M^{c_k}  \prod_{j \in I_{k}} (j-1)!}.
$$
\end{definition}

\noindent Note that $\Gin(N,1,1)$ is $\Gin(N)$. The following identity is essential in the proof of Theorem \ref{PG2}.

\begin{proposition}\label{PG1} For any $g \in L^2(\mu)$, and  $\{ \lambda_i \}_{I_{k}} \sim \Gin(N,M,k)$,
$$ \E \left( \prod_{i \in I_{k}} g \left( \lambda_i N^{M/2}\right)  \right)= \det \left(f_{i,j}\right)_{i,j \in I_{k}} \quad \text{where} \ f_{i,j} = \frac{1}{(j-1)!} \int z^{i-1} \bar{z}^{j-1} g \left( z^M \right) \dd \mu ( z).$$ 
\end{proposition}

\begin{proof} By definition, $\{ N^{M/2} \lambda_i \}_{I_{k}}= \{ N^{M/2} z_i^M \}_{I_{k}}$ where $\{ z_i \}_{I_{k}} \sim \rm{R}(N,M,k)$. The points $\{ N^{M/2} z_i \}_{I_{k}} $ have joint density
$$
\frac{1}{c_k! \prod_{j \in I_{k}} (j-1)!} \prod_{\substack{i<j \\ i,j \in I_k }} |z_i^M - z_j^M|^2  \prod_{i \in I_{k}} |z_i|^{2(k-1)} \dd \mu (z_i),
$$
and, for indices in $\llbracket 1, q \rrbracket$ where $q=|I_k|$,
$$ \prod_{i=1}^q |z_i|^{2(k-1)} \prod_{1 \leq i<j \leq q} |z_i^M - z_j^M|^2 = \det\left({z_i^{M(j-1)+k-1}}\right)_{i,j=1}^q \det\left({\bar{z}_i^{M(j-1)+k-1}}\right)_{i,j=1}^q $$
where the exponents yield exactly the elements of $I_{k}$. Thus we can index by $I_k$, and use Lemma \ref{Andr} with $\phi_i(z) = z^{i-1} g(z)$, $\psi_j(z)= \bar{z}^{j-1}$. This yields
$$
\E \left( \prod_{i \in I_{k}} g (\lambda_i N^{M/2})  \right)  = \E \left( \prod_{i \in I_{k}} g (z_i^M N^{M/2})  \right)  = \frac{c_k!}{c_k! \prod_{j \in I_{k}} (j-1)!} \det \left( \int z^{i-1} \bar{z}^{j-1} g \left( z^M \right) \dd \mu ( z) \right)_{i,j \in I_{k}} 
$$
as was claimed.
\end{proof}

We can now state and prove our main result. The image of the Ginibre point process under any power map is the superposition of independent Power-Ginibre blocks. Below is a more detailed statement of the theorem than the one that was stated in the introduction.

\begin{theorem}[Power-Ginibre decomposition]\label{PG2} We have the equality in distribution
$$ \{\lambda_1^M, \dots, \lambda_N^M\} \stackrel{d}{=} \{ z_1^M, \dots, z_N^M\}$$
where the $(z_i)_{I_{k}}$ are distributed according to $\rm{R}(N,M,k)$ and independently for different values of $k$. In other words, the distribution of the $M$-th Powers of the Ginibre eigenvalues is a superposition of $M$ independent Power-Ginibre point processes:
$$ \rm{Gin}(N)^M \stackrel{d}{=} \bigcup_{k=1}^M \rm{R}(N,M,k)^M \stackrel{d}{=} \bigcup_{k=1}^M \rm{Gin}(N,M,k) .$$
\end{theorem}

\begin{proof} Let $P \in \mathbb{C}[X, \overline{X}]$ and use Corollary \ref{ProStat} with the function $g(z)=P(z^M, \bar{z}^M)$. The matrix is then $M$-striped, with coefficients
$$
f_{i,j} = \frac{1}{(j-1)!} \int z^{i-1} \bar{z}^{j-1} P(z^M, \bar{z}^M) \dd \mu ( z)
$$
when $i$ and $j$ belong to the same progression. We therefore use the factorization from Lemma \ref{Mstriped} and write
$$
\E \left( \prod_{i=1}^N g(N^{M/2} \lambda_i^M, N^{M/2} \bar{\lambda_i}^M) \right) = \det \left(f_{i,j}\right) = \prod_{k=1}^M \det \left(f_{i,j} \right)_{I_{k}}.
$$
These in turn are characterized by Proposition \ref{PG1} as the product statistics of Power-Ginibre,
$$
\E \left( \prod_{i=1}^N g( N^{M/2} \lambda_i^M, N^{M/2} \bar{\lambda_i}^M) \right) = \prod_{k=1}^M  \E \left( \prod_{i \in I_{k}} g( N^{M/2} z_i^M,  N^{M/2} \bar{z_i}^M) \right).
$$
These statistics characterize the distribution of a set of points, as such expressions with polynomial $g$ generate all mixed symmetric polynomials (see Lemma \ref{sympol2}), and the distributions involved are characterized by their moments. The result follows. \end{proof}

Power-Ginibre decomposition encompasses the previously mentioned results about Ginibre. For $M=1$ it yields the original Ginibre point process, and for $M \geq N$ it yields $N$ independent blocks. The joint law of the radii is also coherent with Kostlan's theorem, as can be computed directly from Proposition \ref{PG1}.

\begin{proposition}[Kostlan for Power-Ginibre] The set of Power-Ginibre squared radii $\{ N^M |\lambda_i|^2 \}_{I_{k}} $ is distributed as a set of independent gamma variables, with parameters matching the indices $I_{k}$, to the power $M$.
\end{proposition}

\begin{proof} Let $g \in \mathbb{C}[X]$ and apply Proposition \ref{PG1} with the radially symmetric function $g(|\cdot|^2)$. The matrix is then diagonal, with coefficients
$$
f_{i,i}  = \frac{1}{(i-1)!} \int |z|^{2i-2} g(|z|^{2M}) \dd \mu ( z) 
= \frac{1}{(i-1)!} \int_{r=0}^{\infty} r^{i-1} g(r^M) e^{-r} \dd r 
= \E \left( g(\gamma_{i}^M) \right)
$$
where $i$ is indexed by $I_{k}$. That is to say,
$$
\E \left( \prod_{i \in I_{k}} g( N^{M/2} |\lambda_i|^2) \right) = \prod_{i \in I_{k}} \E \left(  g(\gamma_{i}^M) \right) =  \E \left(  \prod_{i \in I_{k}} g(\gamma_{i}^M) \right). 
$$
These statistics characterize the distribution of a set of points, since such expressions with polynomial $g$ generate all symmetric polynomials by Lemma \ref{sympol}, and all distributions involved are characterized by their moments. This completes the proof. \end{proof}

\subsection{Asymptotic study of the Power-Ginibre Ensembles}\label{Assymp}

We have shown the relevance of the $\Gin(N,M,k)$ blocks in the analysis of the powers of $\Gin(N)$.  These smaller blocks themselves are determinantal and can be studied using standard techniques. In the following paragraphs, we explore the coherence of Theorem \ref{PG2} with the convergence to the Circular Law and the Gaussian Free Field. We also study the Power-Ginibre kernels at microscopic scales.

\subsubsection{Twisted circular law.}
Recall the form of the Ginibre Kernel,
$$
K_N (z,\omega) = \frac{N}{\pi} e^{-\frac{N}{2} (|z|^2 + |\omega|^2)} \sum_{k=1}^N \frac{(Nz\overline{\omega})^k}{k!}.
$$
We will make use of the following fact, the proof of which requires only the Central Limit Theorem and properties of Poisson distributions.
\begin{fact}[Poisson asymptotics]\label{Poissasym} For any $r>0$,
$$ e^{-Nr} \sum_{k=1}^N \frac{(Nr)^k}{k!} \xrightarrow[N \rightarrow \infty]{} f(r) =
\left\{
\begin{array}{ll}
1 & \text{if } r<1 \\
1/2 & \text{if } r=1 \\
0 & \text{if } r>1.
\end{array}
\right.
$$
\end{fact}
This simple threshold property is one way to establish that the empirical measure of Complex Ginibre eigenvalues converges weakly to the circular law, the density of which is discontinuous along the unit circle. That is, for any continuous and bounded $f: \mathbb{C} \rightarrow \mathbb{R}$,
$$
\frac{1}{N} \sum_{i=1}^N f(\la_j) \xrightarrow[N \rightarrow \infty]{a.s.} \frac{1}{\pi} \int_{\D} f(\omega) \dd m(\omega).
$$
Furthermore, using (\ref{CoV}), for any such $f$, if $M$ is fixed and $N$ tends to infinity,
$$ \frac{1}{N} \sum_{j=1}^N f(z_j^M) N \xrightarrow[N \rightarrow \infty]{a.s.} \frac{1}{\pi} \int_{\D} f(z^M) \dd m(z) = \frac{1}{\pi M } \int_{\D} f(\omega) |\omega|^{2/M-2} \dd m(\omega) .$$
The singular weight $\frac{1}{ \pi M} |\omega|^{2/M-2}$ behaves like an approximation to the identity; it is the density on $\D$ we refer to as the $M$-twisted circular law. On the other hand, by the Power-Ginibre decomposition,
$$\frac{1}{N} \sum_{j=1}^N f \left( z_j^M \right) = \frac{1}{N} \sum_{k=1}^M \sum_{i \in I_k} f(\la_j)$$
where the second sum is over independent blocks. The fact that every Power-Ginibre block is determinantal enables us to compute its exact part in the final limit.

\begin{proposition}[Explicit kernel]
The measure $\Gin(N,M,k)$ is a determinantal point process indexed by $I_{k}$ with kernel
\begin{equation}\label{kernelPG}
K_{N,M,k}(z,w) = \Gamma(k) \sum_{l=0}^{c_k-1} \frac{(N^M z \overline{w})^l}{(Ml+k-1)!}
\end{equation}
with respect to the measure
$$
\dd \nu_{N,M,k} (z)= \frac{1}{Z_{\nu_{N,M,k}}} |z|^{2 \frac{k-M}{M}} e^{- {N} |z|^{2/M} } \dd m(z),
$$
where $$ Z_{\nu_{N,M,k}} = \pi N^{-k} M \Gamma(k).$$
\end{proposition}

\begin{proof}
The usual methods for determinantal point processes apply (see \cite{Mehta}), and as the measure is radially symmetric, an orthonormal basis of polynomials is given by the monomials, scaled by their $L^2(\dd \nu_{N,M,k})$ norm. We compute and find
$$
\int_{\mathbb{C}} |z|^{2l} \dd \nu_{N,M,k} = \frac{\Gamma(Ml+k) }{N^{Ml} \Gamma(k)}.
$$
Hence, the kernel is given by formula (\ref{kernelPG}).
\end{proof}

We can now study the convergence of each block to the $M$-twisted circular law. Recall that $c_k=|I_k|$ is the number of points of $\Gin(N,M,k)$ (see (\ref{Progression})).

\begin{proposition}[Mean Twisted Circular Law]\label{meantwist}
The mean density of points of $\Gin(N,M,k)$ converges to the $M$-twisted circular law, that is for any fixed $M \geq k$ and any test function $f$,
$$ \frac{1}{c_k} \E \left( \sum_{i \in I_k} f(\la_j) \right) \xrightarrow[N \rightarrow \infty]{}   \frac{1}{\pi M } \int_{\D} f(\omega) |\omega|^{2/M-2} \dd m(\omega).$$
\end{proposition}

\begin{proof}
It suffices to compute the asymptotic density, obtained directly from the determinantal kernel,
\begin{align*}
\frac{1}{ c_k Z_{\nu_{N,M,k}} } K_{N,M,k}(z,z) |z|^{2 \frac{k-M}{M}} e^{- {N} |z|^{2/M} } & = \frac{N}{\pi c_k M} |z|^{2/M -2} \sum_{l=0}^{c_k-1} \frac{(N |z|^{2/M})^{Ml+k-1}}{(Ml+k-1)!} e^{- {N} |z|^{2/M} } \\
& = \frac{N}{ \pi c_k M} |z|^{2/M -2} e_{M,k}^{(c_k)} (N |z|^{2/M}) e^{- {N} |z|^{2/M} }.
\end{align*}
We then use the asymptotics of the partial sums of the exponential, deduced from formula (\ref{partialsums}) and Proposition \ref{Poissasym}, as well as the fact that $c_k \sim \frac{N}{M}$, to conclude that
$$
\frac{1}{ c_k Z_{\nu_{N,M,k}} } K_{N,M,k}(z,z) |z|^{2 \frac{k-M}{M}} e^{- {N} |z|^{2/M} }  \xrightarrow[N \rightarrow \infty]{} \frac{1}{\pi M} |z|^{2/M-2} \mathds{1}_{\mathbb{D}},
$$
as was claimed.
\end{proof}
\noindent In other words, each Power-Ginibre block contributes equally to first order asymptotics. In fact, one can strengthen this averaged result as follows.
\begin{theorem}[Twisted Circular Law]\label{twist}
For any $k$, the empirical distribution of $\Gin(N,M,k)$ converges weakly to the $M$-twisted circular law, that is for any test function $f$
$$ \frac{1}{c_k} \sum_{i \in I_k} f(\la_j) \xrightarrow[]{a.s.} \frac{1}{\pi M } \int_{\D} f(\omega) |\omega|^{2/M-2} \dd m(\omega).$$
\end{theorem}
A proof of this convergence can be deduced from the Mean Circular Law using canonical arguments, reviewed for instance by Hwang in \cite{Hwang}. For the sake of brevity we do not reproduce the argument here.
\begin{remark}
The above results hold for any fixed $M$, and $N$ going to $\infty$. As we know that $M \geq N$ gives independent points, one could say the parameter $M$ gives an interface between random matrix statistics and independent statistics. 
\end{remark}

\subsubsection{Gaussian Free Field.}
In this section we will consider a smooth function $f : \mathbb{C} \mapsto \mathbb{R}$ with compact support in $\mathbb{D}$. The following was proved in \cite{RiderVirag}.

\begin{theorem}[Rider, Vir\'ag]\label{GFF1}
If we denote by $X_f$ the centered linear statistics of the complex Ginibre ensemble,
$$X_f^{(N)} = \sum_{i=1}^N f(\lambda_i) - \frac{N}{\pi} \int_{\mathbb{D}} f(z) \dd m(z),$$
then $X_f$ converges without renormalization to a gaussian variable, namely:
$$ X_f^{(N)} \xrightarrow[N \rightarrow \infty]{d} \mathscr{N}(0,\sigma_f^2), \quad \text{where } \sigma_f^2 = \frac{1}{4\pi}\int_{\mathbb{D}} |\nabla f(z)|^2 \dd m(z).$$
\end{theorem}
Applying this theorem to the function $f_M =f\circ \pi_M$, which is still smooth and compactly supported in $\D$, we get
$$ X_{f_M}^{(N)} \xrightarrow[N \rightarrow \infty]{d} \mathscr{N}(0,\sigma_M^2), \quad \text{where } \sigma_M^2 = \frac{1}{4\pi}\int_{\mathbb{D}} |\nabla f_M(z)|^2 \dd m(z).$$
Using the identity $|\nabla f_M|^2 = |\pi_M'(z)|^2 |\nabla f (z^M)|^2$, the usual change of variable yields
$$\sigma_M^2 = M \sigma_f^2. $$
On the other hand, using Power-Ginibre decomposition and Theorem \ref{meantwist},
\begin{align*}
X_{f_M}^{(N)} = \sum_{i=1}^N f(\lambda_i^M) - \frac{N}{\pi} \int_{\mathbb{D}} f(z^M) \dd m(z) 
= \sum_{k=1}^M \left( \sum_{i \in I_k} f(z_i) - \frac{c_k}{M \pi} \int_{\mathbb{D}} f(z) |z|^{2/M-2} \dd m(z) \right)
\end{align*}
which is a sum of $M$ centered independent terms, converging to a gaussian variable with variance $M \sigma_f^2$. This simple fact, together with Theorem \ref{meantwist}, suggests that every term converges to a centered gaussian with variance $\sigma_f^2$, which is indeed the case. In order to prove this, we first need to evaluate some more precise asymptotics of the Power-Ginibre Kernel. We consider the following related quantity~:
$$ J_{N,M,k}(z,\omega) := \frac{1}{Z_{\nu_{N,M,k}}^2} |K_{N,M,k}(z, \omega)|^2 |z|^{2 \frac{k-M}{M}}|\omega|^{2 \frac{k-M}{M}} e^{-N |z|^{\frac{2}{M}}-N |\omega|^{\frac{2}{M}}}.$$

For any $\epsilon > 0$ we will denote
$$
\Omega_{\epsilon} = \{ (z, \omega) \in \D^2 \ | \ |\pi - \arg(z \bar{\omega}) | > \epsilon \}.
$$

And $\angle_{M, \epsilon} := \{ (z, \omega) \in \D \ | \ z \in \angle_{M}, \ (z^M, \omega^M) \in \Omega_{\epsilon}\}$.

\begin{lemma}\label{Mroot}
For any $\epsilon >0$, there is a $\delta_1>0$ such that the following holds on $\Omega_{\epsilon}$~:
$$
J_{N,M,k}(z,\omega)
=
 \frac{N^{2} }{\pi^2 M^4} |z \omega|^{2/M -2} e^{-N|z^{\frac{1}{M}}-\omega^{\frac{1}{M}}|^2} (1+O(e^{-N\delta_1})),
$$
where the representatives $z^{\frac{1}{M}}$ and $\omega^{\frac{1}{M}}$ are chosen so as to minimize their distance; and there is a $\delta_2>0$ such that the following holds on $\Omega_{\epsilon}^c$~:
$$
J_{N,M,k}(z,\omega)=O(e^{-N\delta_2}).
$$
\end{lemma}

\begin{proof} If $\zeta$ is a primitive $M$-th root of unity, then the preimage $\pi_M^{-1}({z})$ is stable under multiplication by $\zeta$. We will denote by $z^{\frac{1}{M}}$ a chosen representative, and state precisely where this choice matters, and when it does not. \medskip

It is clear that the expression
$$
{w^{1-k}} {e_{M,k}(w)} = \sum_{j \geq 0} \frac{w^{Mj}}{(Mj+k-1)!}
$$
is stable under multiplication by $\zeta$, and so we can write $ z^{\frac{1-k}{M}} {e_{M,k}(z^{\frac{1}{M}})}$ without possible confusion. Replacing $e_{M,k}$ by a linear combination of exponentials according to (\ref{partialsums}), one gets 
$$
z^{\frac{1-k}{M}} {e_{M,k}(z^{\frac{1}{M}})}
=
\frac{1}{M} \sum_{l=1}^{M} \left( {\zeta^{l-1} z^{\frac{1}{M}}}\right)^{1-k} e^{\zeta^{l-1} z^{\frac{1}{M}}}
=
\frac{1}{M} \sum_{\omega \in \pi_M^{-1}(z)} \omega^{1-k} e^{\omega},
$$
which indeed doesn't depend on the choice of a representative, as it is an average over all representatives. The same can be written about partial sums up to degree $N$. The above expression is then the one we find in the kernel of $\Gin(N,M,k)$~:
$$
K_{N,M,k} (z,\omega) 
= 
\Gamma(k) \sum_{l=0}^{N-1} \frac{(N (z \bar{\omega})^{\frac{1}{M}})^{Ml}}{(Ml+k-1)!}
=
\frac{\Gamma(k)}{M} \sum_{u \in \pi_M^{-1}(z \bar{\omega})} (Nu)^{1-k} e_N(Nu)
$$
The dominant term in the asymptotics will be the choice of $u \in \pi_M^{-1}(z \bar{\omega})$ with the largest real part. This is uniquely defined when $z \bar{\omega} \in \mathbb{D}-\mathbb{R}_-$ and corresponds to $z^{\frac{1}{M}} \omega^{\frac{1}{M}}$ where the two representatives are chosen as to minimize the distance $|z^{\frac{1}{M}}-\omega^{\frac{1}{M}}|$. Combining the kernel with the density of the reference measure, it follows that, on $\Omega_{\epsilon}$, bounding by some $\delta_1>0$ the real part of the difference between $(z \bar{\omega})^{1/M}$ and the other possible choices of the $M$-th root,
\begin{align*}
J_{N,M,k}(z,\omega) := &\frac{1}{Z_{\nu_{N,M,k}}^2} |K_{N,M,k}(z, \omega)|^2 |z|^{2 \frac{k-M}{M}}|\omega|^{2 \frac{k-M}{M}} e^{-N |z|^{\frac{2}{M}}-N |\omega|^{\frac{2}{M}}} \\
& = \frac{N^{2} }{\pi^2 M^4} |z \omega|^{2/M -2} e^{N (z \bar{\omega})^{1/M}+ N (\bar{z} \omega)^{1/M} -N |z|^{\frac{2}{M}}-N |\omega|^{\frac{2}{M}} }(1+O(e^{-N\delta_1}))  \\
& = \frac{N^{2} }{\pi^2 M^4} |z \omega|^{2/M -2} e^{-N|z^{\frac{1}{M}}-\omega^{\frac{1}{M}}|^2} (1+O(e^{-N\delta_1}))
\end{align*}
as was claimed. The bound outside $\Omega_{\epsilon}$ is obtained in the same way.
\end{proof}

We can now state a convergence result for the centered linear statistics of the Power-Ginibre distributions. The gaussian limit is the same as in Theorem \ref{GFF1}, as was claimed. Recall that $c_k$ is the number of points in $\Gin(N,M,k)$ (see (\ref{Progression})).

\begin{theorem}[Gaussian Free Field for Power-Ginibre]\label{GFF2}
For fixed $M,k$, and any smooth $f$ with compact support in $\D$, the centered linear statistics defined by
$$X_{f}^{(N,M,k)} = \sum_{i \in I_k} f(z_i) - \frac{c_k}{M \pi} \int_{\mathbb{D}} f(z) |z|^{2/M-2} \dd m(z)$$ converge without renormalization to a gaussian variable, namely :
$$ X_{f}^{(N,M,k)} \xrightarrow[N \rightarrow \infty]{d} \mathscr{N}(0,\sigma_f^2), \quad \text{where } \sigma_f^2 = \frac{1}{4\pi}\int_{\mathbb{D}} |\nabla f(z)|^2 \dd m(z).$$
\end{theorem}

\begin{proof}
We refer to the general, now well-known, method developped in \cite{CostLeb, AmeurHeden, RiderVirag}. Concretely, we adapt the proof of Theorem~4.1 in \cite{Circ2} in the Ginibre case to the Power-Ginibre case and emphasize what is essentially new. The computation of the variance relies on the explicit formula of the second cumulant,
$$
\mathrm{cum}_N(2) = \frac{1}{2}\int_{\D} \int_{\D} (f(z)-f(\omega))^2 |K_{N,M,k}(z,\omega)|^2 \dd\nu(z) \dd\nu(\omega).
$$
By $z^{1/M}$ we mean any preimage of $z$ by $\pi_M$. The important thing as we will see is that the preimages of $z$ and $\omega$ are chosen in the closest possible way.
\begin{align*}
\mathrm{cum}_N(2) & = \frac{1}{2}\int_{\D^2} (f(z)-f(\omega))^2 |K_{N,M,k}(z,\omega)|^2 \dd\nu(z) \dd\nu(\omega) \\
&= \frac{1}{2}\int_{\D^2} (f(z)-f(\omega))^2 J_{N,M,k}(z,\omega) \dd m (z) \dd m(\omega) \\
&= \frac{1}{2}\int_{\Omega_{\epsilon}} (f(z)-f(\omega))^2 \frac{N^{2} }{\pi^2 M^4} |z \omega|^{2/M -2} e^{-N|z^{\frac{1}{M}}-\omega^{\frac{1}{M}}|^2} \left(1+O(e^{-N\delta_1})\right) \dd m (z) \dd m(\omega) +O(e^{-N\delta_2}) \\
&= \frac{N^{2} }{2 \pi^2} \int_{\angle_{M,\epsilon}} (f(z^M)-f(\omega^M))^2 e^{-N|z-\omega|^2} \dd m (z) \dd m(\omega) \left(1+O(e^{-N\delta_1})\right) +O(e^{-N\delta_2}) \\
&= \frac{N^{2} }{2 \pi^2 M} \int_{\Omega_{\epsilon}} (f(z^M)-f(\omega^M))^2 e^{-N|z-\omega|^2} \dd m (z) \dd m(\omega) \left(1+O(e^{-N\delta_1})\right) +O(e^{-N\delta_2}).
\end{align*}
This integral is the same that appears in the proof of Theorem~4.1 of \cite{Circ2}, applied to the function $f_M$. We can state with $\delta= \min(\delta_1, \delta_2)$~:
$$\mathrm{cum}_N(2)
=  \frac{1}{4 \pi M} \int_{\D} |\nabla f_M|^2 + O(e^{-N\delta})
= \frac{1}{4 \pi} \int_{\D} |\nabla f|^2 + O(e^{-N\delta}).
$$

It remains to prove that other cumulants vanish. It can be checked, using expressions from the proof Theorem~\ref{meantwist} that $\mathrm{cum}_N(1) \xrightarrow[N \rightarrow \infty]{} 0$. For higher cumulants, one could follow the method of \cite{Circ2} in order to bound $\mathrm{cum}_N(l)$ directly when $l \geq 3$. For the convenience of the reader, we present another way to conclude. The variables are centered and have bounded variances, so the families of their distributions are tight. By Theorem \ref{GFF1} and Power-Ginibre decomposition, any converging subsequences of these families are such that their independent limits sum up to a gaussian variable. By Cram\'er's theorem, each of these limits is a gaussian variable. Therefore, every centered linear statistics do converge to a gaussian variable, as was claimed.
\end{proof}

We have used the determinantal structure of the Power-Ginibre processes to establish convergence of the linear statistics to the twisted circular law, and convergence of the centered statistics to the Gaussian Free Field. Note that neither the mean, nor the variance of the linear statistics depend on the parameter $k$ in the limit. Moreover, the variance does not depend on the power $M$ : the limit Gaussian Free Field is the same as for the usual complex Ginibre ensemble.


\subsubsection{Microscopic analysis of the Power-Ginibre kernels.}
We have seen above that the parameter $k$ did not impact the first and second order asymptotic properties of the properly scaled Power-Ginibre distributions. However, it does appear in the microscopic limit. \medskip

In the following, we assume that a primitive $M$-th root of unity $\zeta$ has been chosen, as well as a determination of the $M$-th root for $z, \omega$. The result does not depend on these arbitrary choices.
\begin{proposition}
The unscaled process $N^{M/2} \Gin(N,M,k)$ converges to a determinantal point process, whose kernel is given by the average
\begin{equation}\label{kernelzoom}
\mathscr{K}_{M,k}(z,w) = \frac{1}{M} 
\sum_{j=0}^{M-1} \zeta^{j(1-k)} 
\exp \left( - \frac{1}{2} | \zeta^j z^{\frac{1}{M}} - \bar{\omega}^{\frac{1}{M}}|^{2} + i \Im ( \zeta^{j} z^{\frac{1}{M}} \bar{\omega}^{\frac{1}{M}}) \right)
\end{equation}
with respect to the twisted measure $\frac{1}{\pi M} 
|z|^{\frac{2}{M}-2} \dd m(z)$ on $\mathbb{C}$.
\end{proposition}

\begin{proof}
Proposition \ref{kernelPG} tells us that the process $N^{M/2} \Gin(N,M,k)$ is a determinantal point process with kernel
$$
 K_{N,M,k} \left(\frac{z}{\sqrt{N}},\frac{w}{\sqrt{N}} \right) = \Gamma(k) \sum_{l=0}^{c_k-1} \frac{(z \overline{w})^l}{(Ml+k-1)!}
$$
with respect to the measure
$$
\dd \nu_{N,M,k} \left( \frac{z}{\sqrt{N}} \right)= \frac{1}{\pi M \Gamma(k)} |z|^{2 \frac{k-M}{M}} e^{- |z|^{2/M} } \dd m(z).
$$
The parameter $N$ is only reflected in $c_k$, that is a deterministic number of order $N/M$. For a given choice of the $M$-th roots $z^{\frac{1}{M}}$, $\omega^{\frac{1}{M}}$, and $\zeta$ a primitive root of unity, when $M,k$ are fixed and $N \rightarrow \infty$, the above converges to a determinantal point process with kernel
$$
\tilde{K}_{M,k}(z,w) = \frac{1}{\pi M} \sum_{l\geq 0}\frac{(z \overline{w})^l}{(Ml+k-1)!} (z \bar{\omega})^{\frac{k-M}{M}} e^{- \frac{1}{2} \left( |z|^{2/M} + |\omega|^{2/M} \right)}
$$
with respect to the Lebesgue measure on $\mathbb{C}$. We now transform this expression using the definitions and identities involved in the proof of Lemma \ref{Mroot}.  It becomes
$$
\tilde{K}_{M,k}(z,w) = \frac{1}{\pi M^2} 
\sum_{j=0}^{M-1} \zeta^{j(1-k)} (z \bar{\omega})^{\frac{1-M}{M}} 
e^{\zeta^j (z \bar{\omega})^{\frac{1}{M}}}
e^{- \frac{1}{2} \left( |z|^{2/M} + |\omega|^{2/M} \right)}
$$
The difference with the scaled case is that all term of this sum have a non trivial contribution. We write~:
$$
\tilde{K}_{M,k}(z,w) = \frac{1}{\pi M^2} 
(z \bar{\omega})^{\frac{1-M}{M}} 
\sum_{j=0}^{M-1} \zeta^{j(1-k)} 
\exp \left( - \frac{1}{2} \left( | z^{\frac{1}{M}}|^2 + |\omega^{\frac{1}{M}}|^2 - 2 \zeta^{j} z^{\frac{1}{M}} \bar{\omega}^{\frac{1}{M}}) \right) \right)
$$
The result follows when the reference measure is the pushforward of the Lebesgue measure by $\pi_M$.
\end{proof}

\section{Generalization and partial results}\label{Generalization}

Our proof technique does not depend on the reference measure, but only relies on quadratic repulsion, provided that the distributions involved are characterized by their moments. As we will see, the assumption that all moments are finite is not essential and can be weakened. We give below a generalization of Theorem \ref{PG1} to the more general case of two-dimensional radially symmetric beta ensembles with $\beta =2$ and a suitable potential. We then look specifically at the CUE case. Subsections \ref{GUE} and \ref{BetaDec} treat the cases when the potential lacks full radial symmetry (as in the GUE case) and when $\beta$ is a higher even integer, respectivelly.

\subsection{General Power decomposition}\label{PVDec}
We consider a reference measure given by a radial external potential
$$\dd \mu_V(z) = \frac{1}{Z_V} e^{- V \left( |z|^2 \right)} \dd m ( z),$$ 
where the function $V : \R_+ \mapsto \R \cup \{ \infty \}$ is such that

\begin{equation}\label{finiteness}
\prod_{1 \leq i<j \leq N} |z_i-z_j|^2 e^{-\sum_{i=1}^N V(|z_i|^2)}  \ < \ \infty.
\end{equation}
When properly normalized, this is a probability density. Note that, while we sometimes call $V$ the potential, strictly speaking the potential is given by $V(|z|^2)$, such that the quadratic potential case corresponds to $V= \rm{Id_{\R_+}}$. Condition (\ref{finiteness}) is enough for all definitions and results below to hold.

\begin{definition}\label{Vgammafun}
We denote by $\VGamma$ the analog of the $\Gamma$ function with potential $V$,
$$\VGamma(\alpha) = \int_{0}^{\infty} t^{\alpha-1} e^{-V(t)} \dd t.$$
As long as $\alpha$ is such that the above is finite, we define the $\VGamma$ distribution of parameter $\alpha$, denoted by $\gamma(V,\alpha)$, by its density on $\mathbb{R}_+$, \begin{equation}\label{Vgamma}
\frac{1}{\VGamma(\alpha)} t^{\alpha-1} e^{-V(t)} \mathds{1}_{\mathbb{R}_+}.
\end{equation}
\end{definition}

\begin{definition}\label{Betadef} We say that the complex random points $(\lambda_1, \dots, \lambda_N)$ are distributed according to the beta ensemble with parameter $\beta$ and radial potential $V$ if they have joint density
$$
\frac{1}{\pi^N Z_{\beta,N}} \prod_{i<j} |z_i-z_j|^{\beta} \prod_{i=1}^N e^{- V(|z_i|^2)}
$$
with respect to the Lebesgue measure on $\C^N$.
\end{definition}

\noindent We denote by $\V(N)$ or $\V_{(V)}(N)$ the $\beta$-ensemble with potential $V$ and $\beta=2$. Its density is given by
$$ \frac{1}{N! \prod_{j=1}^N \VGamma(j)} \prod_{i<j} |\lambda_i-\lambda_j|^2  \dd \mu_V^N ( \mathbf{\bla}). $$
Such distributions can be achieved by radially symmetric ensembles of normal matrices, but, as in the previous sections, we shall not use the underlying matrix structure.

\begin{corollary}[Product Statistics]\label{ProstatV} Let $E=\mathbb{C}$, $g \in L^2(\mu)$, and  $\{ \lambda_1, \dots, \lambda_N\}$ be $\V(N)$ points. Then
$$ \E \left( \prod_{k=1}^N g (\lambda_k)  \right)= \det [f_{i,j}]_{i,j = 1}^{N} \quad \text{where} \ f_{i,j} = \frac{1}{\VGamma(j)} \int z^{i-1} \bar{z}^{j-1} g(z) \dd \mu_V( z).$$ 
\end{corollary}

\begin{proof}
The proof is the same as in the Ginibre case, {\it mutatis mutandis}.
\end{proof}

\begin{definition} For any triple $(N,M,k)$ such that $ 1 \leq k \leq M \leq N $ we define the Root distribution $\rm{R}_V(N,M,k)$ as the point process indexed by $I_{k}$ with joint density
$$ \frac{1}{Z_{R_V(N,M,k)}} \prod_{\substack{i<j \\ i,j \in I_k }} |z_i^M - z_j^M|^2 \prod_{i \in I_{k}} |z_i|^{2(k-1)} \dd \mu_V (z_i)$$
where
$$
Z_{R_V(N,M,k)}= c_k! \prod_{j \in I_{k}} \VGamma(j),
$$
and the Power distribution $\V(N,M,k)$ as the image of the root distribution under the power map $ z \mapsto z^M $, that is the point process indexed by $I_{k}$ with joint density
$$ \frac{1}{Z_{V(N,M,k)}} \prod_{\substack{i<j \\ i,j \in I_k}} |z_i - z_j|^2  \prod_{i \in I_{k}} |z_i|^{\frac{k-M}{M}}  e^{- V(|z_i|^{2/M})} \dd m(z_i),$$
where
$$
Z_{V(N,M,k)}=c_k! \ Z_V^{c_k} M^{c_k} \prod_{j \in I_{k}} \VGamma(j).
$$
\end{definition}

It is clear that $\V(N,M,k)$ is again a beta distribution with $\beta=2$ and external potential
$$W_{M,k} (x) = V(x^{1/M}) -  \frac{k-M}{M} \log(x), $$
so that up to re-indexation we can write $\V(N,M,k)= \V_{(W_{M,k})} (c_k)$. \medskip

\begin{theorem}[General Power decomposition] We have the equality in distribution
$$ \{\lambda_1^M, \dots, \lambda_N^M\} \stackrel{d}{=} \{ z_1^M, \dots, z_N^M\}$$
where the $(z_i)_{I_{k}}$ are distributed according to $\rm{R}_V(N,M,k)$ and independent for different values of $k$. In other words, the distribution of the $M$-th powers of $\V(N)$ is a superposition of $M$ independent Power distributions:
$$ \V(N)^M \stackrel{d}{=} \bigcup_{k=1}^M \rm{R}_V(N,M,k)^M \stackrel{d}{=} \bigcup_{k=1}^M \V(N,M,k).$$
\end{theorem}

This implies in particular two other analogous results, that were known from the work of Hough, Krishnapur, Peres and Vir\'ag \cite{HKPV}, namely, independence when $M \geq N$, and a version of Kostlan's independence of radii with $\VGamma$ distributions. The fact that a version of Kostlan's Theorem holds for any potential has been used by Chafa\"\i \ and P\'ech\'e in \cite{ChafaiPeche} to prove a limit theorem on the edge.

\begin{proof} We first cut-off the potential $V$, and replace it with the following potential, to ensure that all moments are finite.
$$ V_{N, \eps} (x) = \max (V_N(x), \eps x)  $$
The technique we used in the proof of Theorem \ref{PG2} did not rely on the potential, provided all distributions were characterized by their moments. Thus, the result holds for $\V_{(V_{N,\eps})}(N)$. That means that for any continuous and bounded $f$, if we denote by $\lambda_i^{(\eps)}$ the eigenvalues of $\V_{(V_{N,\eps})}(N)$ and by $(\mu_i^{(\eps)})_{i \in I_k}$ those of $\V_{(V_{N,\eps})}(N,M,k)$,
$$
\E \left( \prod_{i=1}^N f(\lambda_i^{(\eps)^M}) \right) = \prod_{k=1}^M \E \left( \prod_{i \in I_{k}} f(\mu_i^{(\eps)}) \right).
$$
Because of the finiteness condition (\ref{finiteness}), dominated convergence holds when $\epsilon \rightarrow 0$. This yields the result for $\V_{(V)}(N)$.
\end{proof}

There are several relevant examples of such distributions. Here are some of these.
\begin{itemize}
\item{\bf Products of complex Ginibre matrices.} As shown in \cite{AkeBurda}, the eigenvalue distribution of $G_1 \dots G_k$ where $G_1, \dots, G_k$ are independent Ginibre matrices is given by the beta ensemble with $\beta=2$ and potential $V_k (|z|^2)=-\ln w_k (z)$, where
$$
w_1(z) = e^{- |z|^2}, \quad w_{m+1} (z) = 2 \pi \int_{0}^\infty w_m \left(\frac{z}{ r}\right) e^{-r^2} \frac{\dd r}{r}
$$
which gives by induction
$$
w_{m+1}(z) = (2 \pi)^m \int_{r_1=0}^{\infty} \dots \int_{r_m = 0}^{\infty} e^{\frac{|z|^2}{ r_1 \dots r_m} - r_1^2 - \dots r_m^2} \frac{\dd r_1 \dots \dd r_m}{ r_1 \dots r_m}.
$$
This fact seems related to the idea that products of independent Ginibre should behave like Ginibre powers in several respects, for which arguments are provided in \cite{Burda}.

\item {\bf Truncated Unitary Ensembles.} $N \times N$ Minors of the Circular Unitary Ensemble of size $N+n$ have been shown in \cite{Sommers} to have eigenvalue density proportional to
$$
\prod_{k=1}^N (1-|z_k|^2)^{n-1} \mathbf{1}_{|z_k|<1} 
\prod_{1 \leq i<j \leq N} |z_i-z_j|^2.
$$
In that case, the $\VGamma$ variables are usual beta variables. Namely, the set of radii is distributed as a set of independent variables with distribution $\beta_{1,n}, \beta_{2,n} \dots, \beta_{N,n}$.


\item {\bf Spherical ensemble.} This ensemble corresponds to the distribution of eigenvalues of $G_1^{-1} G_2$ where $G_1, G_2$ are i.i.d. Ginibre matrices of size $N$. The eigenvalue density is then proportional to
$$
\prod_{k=1}^N \frac{1}{(1+|z_k|^2)^{N+1}}
\prod_{1 \leq i<j \leq N} |z_i-z_j|^2
$$
as shown in \cite{Krishnapur}. This is a case where all moments are not defined.
\end{itemize}

\subsection{Circular Unitary Ensemble}
\subsubsection{Rains' decomposition for the CUE powers.}\label{CUE}
It is clear that we could have considered more general beta ensembles, replacing $e^{-V}$ by any suitable measure $\mu$, possibly including atoms, or supported on lower-dimensional manifolds. A famous example is the CUE case, that was treated by Rains in \cite{Rains}. The reference measure $\mu$ is then the uniform measure on the unit circle, and $\beta=2$. For the reader's convenience, we reformulate this result here with our conventions. \medskip

The product statistics now take the following form, following from Lemma \ref{Andr}.
\begin{corollary}[Product Statistics]\label{ProstatCUE} Let $g \in L^2(\mu)$, and  $\{ e^{i \theta_1}, \dots, e^{i \theta_N} \}$ be $\CUE(N)$ points. Then
$$ \E \left( \prod_{k=1}^N g (e^{i \theta_k})  \right)= \det \left( f_{j,k} \right)_{j,k = 1}^{N} \quad \text{where} \ f_{j,k} = \frac{1}{ 2 \pi } \int e^{i(j-k)\theta}  g(e^{i \theta}) \dd \mu ( \theta).$$ 
\end{corollary}
\noindent From this, one can derive the formula of Heine-Szeg\H{o}, which yields a Toeplitz matrix.
\begin{corollary}[Heine-Szeg\H{o} ] Let $g= \sum_{-K}^K a_j X^j$ a Laurent polynomial, and  $\{ e^{\theta_1}, \dots, e^{\theta_N} \}$ be $\CUE(N)$ points. Then
$$ \E \left( \prod_{k=1}^N g (e^{i \theta_k})  \right)= \det \left( a_{i-j} \right)_{i,j = 1}^{N}$$ 
\end{corollary}
\noindent Kostlan's theorem now holds in a trivial way. Independence of high powers had been first proved by Rains in \cite{Rains1}; a more general proof of it is given in \cite{HKPV}.
\begin{theorem}[Rains] For any integer $M \geq N$, the set $\{ e^{i M\theta_1}, \dots,e^{i M\theta_N} \} $ is distributed as a set of independent variables with uniform arguments.
\end{theorem}
\noindent The approach developed above yields a new general proof of Rains' decomposition, as the independent blocks obtained are easily identified as smaller $\CUE$ blocks.
\begin{theorem}[Rains] For any $M$, we have the equality in distribution:
$$ \CUE(N)^M  \stackrel{d}{=} \bigcup_{k=1}^M \CUE(I_{N,M,k}). $$
\end{theorem}

\begin{proof}
The root distributions now have joint density proportional to
$$
\prod_{\substack{j<k \\ j,k \in I_k} } |e^{i M \theta_j}-e^{i M \theta_k}|^2
$$
as all eigenvalues have radius $1$, and therefore the {  power} distribution obtained in the end is another, smaller, $\CUE$ distribution. The size of these blocks are the cardinalities of the progressions $I_{k}$,
$$
c_k= |I_k| = \Big\lceil \frac{N -k}{M} \Big\rceil,
$$
which correspond to the ones given in \cite{Rains}.
\end{proof}

\subsubsection{The characteristic polynomial of a unitary matrix.}\label{CUEPol}
Another stunning property of the CUE is the fact that its characteristic polynomial on the unit circle,
$$ Z=P_{U_N}(1) = \det (I-U),$$
is distributed like a product of independent random variables. This result was first proved in \cite{BHNY} using an explicit decomposition of the Haar measure. The moments of this polynomial had been computed before by Keating-Snaith in \cite{KeatingSnaith} using Selberg integral. It turns out that the above methods and identities give a somewhat more straightforward proof.

\begin{lemma}[Translation invariance]\label{transinv} For any measure $\mu$ on $\mathbb{C}$, with $g \in L^2(\mu)$, let us define :
$$ \forall z_1,z_2 \in \mathbb{C} \qquad f_{i,j} (z_1,z_2) : = \int_{\C} (\lambda-z_1)^{i} (\overline{\lambda -z_2})^{j}  g(\lambda) \mu(\dd \lambda).$$ 
Then $\det \left( f_{i,j} (z_1,z_2) \right)_{i,j=0}^N$ is a constant function of $z_1,z_2$.
\end{lemma}
\begin{proof} By Lemma \ref{Andr},
$$
\det \left( f_{i,j} (z_1,z_2) \right)_{i,j=0}^N = \int_{\mathbb{C}^N} \prod_{k=1}^N g(\lambda_k) \ \det \left((\lambda_i - z_1)^{j-1}\right) \ \overline{\det\left((\lambda_i - z_2)^{j-1}\right)} \mu(\dd \lambda_1) \dots \mu (\dd \lambda_N).
$$
Since the Vandermonde determinant is invariant by translation, we have
$$ \det \left((z_1-\lambda_i)^{j-1}\right) = \det\left((-z_1+\lambda_i)^{j-1}\right) = \det\left((\lambda_i)^{j-1}\right),$$
which proves the claim.
\end{proof}

\begin{proposition}\label{momcue} The complex moments of $Z=P_{U_N}(1)$ are given by the following minor of the symmetric Pascal matrix,
$$\E \left( Z^{m} {\bar{Z}}^{n} \right) = \det \left( \binom{i+m+j+n-2}{i+m-1} \right)_{i,j=1}^{N}.$$
Explicit computation of this minor gives:
$$
\E \left( Z^{m} \overline{Z}^{n} \right) = \prod_{k=0}^{N-1} \frac{k! (k+m+n)!}{(k+m)! (k+n)!}.
$$
\end{proposition}
\begin{proof}
The first equality comes from the Corollary \ref{ProstatCUE} with $g(\theta)= (1-e^{i \theta})^m (1-e^{-i \theta})^n$ and translation invariance, Lemma \ref{transinv}. Indeed, for all $a,b \in \llbracket 1,N \rrbracket$,
\begin{align*} f_{a,b} 
& = \frac{1}{2 \pi} \int_{0}^{2 \pi} (1-e^{i \theta})^{m+a-1} (1-e^{-i \theta})^{n+b-1} \dd \theta \\
&= \frac{1}{ 2 \pi} \int_{0}^{2 \pi} \sum_{k=1}^{m+a} \sum_{l=1}^{n+b} \binom{m+a-1}{k-1}  \binom{n+b-1}{l-1} (-e^{i\theta})^{k-l} \dd \theta \\
&= \sum_{k=1}^{(m+a) \wedge (n+b)} \binom{m+a-1}{k-1}  \binom{n+b-1}{k-1} 
 = \binom{m+a+n+b-2}{m+a},
\end{align*} 
where the last equality is a common combinatorial identity, that yields a coefficient of Pascal's matrix, as was claimed. \medskip

To compute this determinant, we first  translate the minor by  multiplying according to lines and columns,
$$
\det \left( \binom{i+m+j+n-2}{i+m-1} \right)_{i,j=1}^{N} = \prod_{k=0}^{N-1} \frac{k! (k+m+n)! }{ (k+m)! (k+n)!}  \det \left( \binom{i+j+m+n-2}{i-1} \right)_{i,j=1}^{N}.
$$
Applying Corollary \ref{ProstatCUE} with $g(\theta)= (1-e^{i \theta})^{m+n}$, we have
$$
\det \left( \binom{i+j+m+n-2}{i-1} \right)_{i,j=1}^{N} = \E \left( \prod_{k=1}^N (1-e^{i \theta_k})^{m+n} \right) = \det\left( \frac{1}{2 \pi} \int_{0}^{2 \pi} e^{i (a-b) \theta} (1-e^{i \theta})^{m+n} \dd \theta  \right)_{a,b}.
$$
This last matrix being upper triangular with a diagonal of ones, its determinant is $1$, hence the result.
\end{proof}

Thus, we recover the moments computed in \cite{KeatingSnaith} and \cite{BHNY} with different techniques. These moments are known to be related to beta distributions in the following way.
\begin{lemma}
For any $k \in \N$, the following formula holds
$$
\E \left( (1+ \sqrt{\beta_{1,k}} e^{i \omega_k})^{m} \overline{(1+ \sqrt{\beta_{1,k}} e^{i \omega_k})^{n}} \right) = \frac{\Gamma(k) \Gamma(k+m+n)}{\Gamma(k+m) \Gamma(k+n)},
$$
where the variables $\omega_k, \beta_{1,k}$ are independent, the omega variables being uniform on $[0,2\pi]$, and the parameters of the beta distributions being given by their indices.
\end{lemma}
For a proof of this Lemma, see \cite{BHNY}. One deduces from it a proof of the decomposition of the distribution of the characteristic polynomial as a product of independent variables.
\begin{theorem}[Bourgade-Hughes-Nikeghbali-Yor] The characteristic polynomial of the $\CUE$ is distributed like a product of independent variables,
$$ Z \stackrel{d}{=} \prod_{k=1}^N (1+ \sqrt{\beta_{1,k}} e^{i \omega_k}).$$
\end{theorem}

\subsection{Partial symmetry and the Gaussian Unitary Ensemble}\label{GUE}
The same technique we have used on Ginibre and the CUE can be applied to processes with partial symmetry, such as the real line, or any star set -- that is, the preimage of $\mathbb{R}_+$ under the power map $\pi_M$. In this case, block decomposition does not hold for all powers, due to the loss of radial invariance, but only for divisors of the number $M$ characterizing the symmetry. On the real line, the symmetry group is $\Z_2$. The point of interest is therefore when $M=2$. It yields a decomposition in two independent blocks, that holds for any symmetric potential. In \cite{EdelmanLacroix}, Edelman and La Croix mentionned this as a natural generalization of their result for the GUE. We finally derive from these methods another result first established in \cite{EdelmanLacroix}, the decomposition of the law of the determinant of the GUE as a product of independent chi-squared variables.
\subsubsection{Power decomposition for processes on the real line.}
We consider a symmetric measure on the real line with density
$$
\dd \mu_V(x) = \frac{1}{Z_V} e^{- V \left( x^2 \right)},
$$ 
with respect to the Lebesgue measure, where the potential $V$ is chosen so that this measure has finite moments. These moments define the $\VGamma$ function, as in Definition \ref{Vgammafun}. We also assume, as before, that all distributions involved are characterized by their moments. We denote the $\beta$-ensemble on $\R$ with potential $V$ and $\beta=2$ by $\V_{\R}(N)$ . Its density is given by
$$ \frac{1}{N! D_V^{(N)}} \prod_{i<j} |x_i-x_j|^2  \dd \mu_V^{\otimes N} ( \bx )$$
where
$$D_V^{(N)}= \det \left( \VGamma\left(\frac{i+j-1}{2}\right) \right)_{i,j \in I_{N,2,1}} \det \left( \VGamma \left(\frac{i+j-1}{2} \right) \right)_{i,j \in I_{N,2,2}} = D_V^{(N,1)} D_V^{(N,2)}.$$
We compute the values of the constants $D_V^{(N)}, D_V^{(N,1)}, D_V^{(N,2)}$ in the case of the GUE in the proof of Proposition \ref{GUEmom}. \vspace{0.5cm}

The proofs of the following results all mimic Section \ref{PGDec}, {\it mutatis mutandis}. Notably, there seems to be no analog of Kostlan's theorem, nor independence of large powers.

\begin{corollary}[Product Statistics]\label{ProstatGUE} Let $E=\mathbb{C}$, $g \in L^2(\mu)$, and  $\{ \lambda_1, \dots, \lambda_N\}$ be $\V_{\R}(N)$ points. Then
$$ \E \left( \prod_{k=1}^N g (\lambda_k)  \right)= \frac{  1  }{D_V^{(N)}} \det \left(f_{i,j}\right)_{i,j = 1}^{N} \quad \text{where} \ f_{i,j} = \int_{\mathbb{R}} x^{i+j-2} g(x) \dd \mu_V(x).$$ 
\end{corollary}

\begin{definition} For $k=1,2$ we define the root distribution $\rm{R}_{V, \mathbb{R}}(N,k)$ as the point process on $\mathbb{R}$ indexed by $I_{N,2,k}$ with joint density
$$
\frac{1}{c_k! D_V^{(N,k)}} \prod_{\substack{i<j \\ i,j \in I_k }} |x_i^2 - x_j^2|^2   \prod_{i \in I_{k}} x_i^{2(k-1)} \dd \mu_V (x_i),
$$
and the Power distribution $\V_{\R}(N,k)$ as the image of the root distribution under the power map $\pi_2$. In other words, $\V_{\R}(N,k)$ is the point process on $\mathbb{R}_+$, indexed by $I_{k}$, with joint density
$$
\frac{1}{c_k! D_V^{(N,k)}} \prod_{\substack{i<j \\ i,j \in I_k }} |x_i - x_j|^2  \prod_{i \in I_{k}} x_i^{k-3/2}  e^{- V(x_i)} \dd x_i.
$$
\end{definition}

\begin{theorem}[Edelman, La Croix] We have the equality in distribution
$$
\{\lambda_1^2, \dots, \lambda_N^2 \} \stackrel{d}{=} \{ x_1^2, \dots, x_N^2\},
$$
where the $(x_i)_{I_{k}}$ are distributed according to $\rm{R}_{V, \mathbb{R}}(N,k)$ and independent for different values of $k$. In other words, the distribution of the squares of $\V_{\R} (N)$ is a superposition of $2$ independent Power distributions:
$$ \V_{\R} (N)^2 \stackrel{d}{=} \rm{R}_V(N,1)^2 \cup \rm{R}_V(N,1)^2 \stackrel{d}{=} \V_{\R}(N,1) \cup \V_{\R}(N,2).$$
\end{theorem}
\noindent For the quadratic potential, $\V_{\R} (N)$ is the distribution of GUE eigenvalues, and the two Power distributions corresponds to Laguerre-Wishart ensembles with half-integer parameters. That is, one recovers the decomposition of Edelman and La Croix introduced in \cite{EdelmanLacroix} for the singular values of the GUE.

\subsubsection{The determinant of the GUE.}
Let $H_N$ be a $\mathrm{GUE}$ matrix, and 
$$P_{H_N}(x) = \det \left( H_N - x \mathrm{Id} \right).$$
One consequence of Edelman and La Croix's results on the GUE is that the absolute value of the determinant of a GUE matrix (that is, $|P_{H_N}(0)|$) is distributed like a product of independent $\chi^2$ variables with explicit parameters. This is a striking similarity with the CUE case reviewed in Subsection \ref{CUEPol}, and as above we give a proof of this result relying on Pascal matrices. The potential is now the usual quadratic one, $V(x)=-\frac{x^2}{2}$, and we denote the normalization constants simply $D^{(N)}, D^{(N,1)}$ and $D^{(N,2)}$. We also recall the following definition: for $k=1,2$,
$$
I_k = \{ i\in [\![ 1,N ]\!] \ | \ i =k \mod 2 \}, \qquad c_k = |I_k|.
$$

\begin{proposition}\label{GUEmom} The moments of $\Pi=P_{H_N}(0)$ are given by the expressions
\begin{align*}
& \E \left( \Pi^{2m+1} \right) =0, \\
\text{and} \qquad 
& \E \left( \Pi^{2m} \right) = 2^{m N} \prod_{l_1=1}^{c_1} \frac{ \Gamma(l_1+m-1/2)}{ \Gamma(l_1-1/2) } \prod_{l_2=1}^{c_2} \frac{\Gamma(l_2+m+1/2)}{\Gamma(l_2+1/2) }.
\end{align*}

\end{proposition}
\begin{proof} The odd moments of $\Pi$ are zero by symmetry. To compute the even moments, we apply Corollary \ref{ProstatGUE} with $g(x)=x^{2m}$.
$$
\E \left( \Pi^{2m} \right) = \E \left( \prod_{k=1}^N \lambda_k^{2m}  \right)  = \frac{  1  }{D^{(N)}} \det \left( \int_{\mathbb{R}} x^{i+j-2} x^{2m} e^{-x^2/2} \dd x \right)_{i,j = 1}^{N}
$$
This determinant splits into two blocks, and we use that
$\int_{\mathbb{R}} t^{\alpha} e^{-\frac{t^2}{2}} \dd t = 2^{\frac{\alpha+1}{2}} \Gamma( \frac{\alpha+1}{2})$
to find
\begin{align*}
& \frac{1}{D^{(N,1)}} \det \left( \int_{\mathbb{R}} x^{i+j-2} x^{2m} e^{-x^2/2} \dd x \right)_{i,j \in I_1} \frac{1}{D^{(N,2)}} \det \left( \int_{\mathbb{R}} x^{i+j-2} x^{2m} e^{-x^2/2} \dd x \right)_{i,j \in I_2} \\
& =  \frac{1}{D^{(N,1)}} \det \left( 2^{ \frac{2a+2b+2m-3}{2}} \Gamma\left( \frac{2a+2b+2m-3}{2} \right) \right)_{a,b=1}^{c_1} \\
& \hspace{3cm} \times \frac{1}{D^{(N,2)}} \det \left( 2^{ \frac{2a+2b+2m-1}{2}} \Gamma\left(  \frac{2a+2b+2m-1}{2} \right) \right)_{a,b=1}^{c_2} 
\end{align*}
where the elements $(i,j) \in I_1$ have been written as $(2a-1, 2b-1)$ and those of $I_2$ as $(2a,2b)$. We have seen in the proof of Proposition \ref{momcue} that any minor of Pascal's matrix of the type
$$ \left( \binom{i+j+m-2}{i-1} \right)_{i,j=1}^{N},$$
where $m$ is any non-negative integer, has determinant $1$. Therefore, 
$$ Q(x) = \det  \left( \frac{ \Gamma(i+j+x-1)}{\Gamma(i) \Gamma(j+x)} \right)_{i,j=1}^{N} $$
is a polynomial in $x$ taking the value $1$ infinitely often. As a consequence it is constant equal to $1$, and for any real parameter $x$,
$$  \det  \left( { \Gamma(i+j+x-1)} \right)_{i,j=1}^{N} = \prod_{i=1}^N \Gamma(i) \prod_{j=1}^N \Gamma(j+x).$$
This gives us the value of the normalization constants
\begin{align*}
D^{(N,1)} & = 2^{c_1^2 -\frac{c_1}{2}} \prod_{k=1}^{c_1} \Gamma(k) \Gamma(k-\frac{1}{2}) \\
\text{and} \qquad D^{(N,2)} & =2^{c_2^2 + \frac{c_2}{2}} \prod_{k=1}^{c_2} \Gamma(k) \Gamma(k+\frac{1}{2}),
\end{align*}
and also yields the expected formula for the moments. Since $c_1+c_2=N$, we have
$$
\E \left( \Pi^{2m} \right) = 2^{m(c_1+c_2)} \prod_{l_1=1}^{c_1} \frac{ \Gamma(l_1+m-1/2)}{ \Gamma(l_1-1/2) } \prod_{l_2=1}^{c_2} \frac{\Gamma(l_2+m+1/2)}{\Gamma(l_2+1/2) }.
$$
This concludes the proof.
\end{proof}

We deduce from these moments a direct proof of the decomposition of the determinant of the GUE as a product of independent variables.

\begin{theorem}[Edelman, La Croix] The value of the GUE characteristic polynomial at $0$ is distributed like the product of independent variables,
$$
P_{H_N}(0) =  \prod_{k=1}^N \lambda_k \stackrel{d}{=} (-1)^{\epsilon} \prod_{k=1}^N \chi^2({2 \lfloor k/2 \rfloor+1}), 
$$
 where $\epsilon$ is a Bernoulli random variable of parameter $p=\frac{1}{2}$, and $\chi^2(m)$ are chi-squared variables with $m$ degrees of freedom. 
\end{theorem}

\begin{proof}
The chi-squared distribution with $k$ degrees of freedom has density
$$
\frac{1}{2^{\frac{k}{2}} \Gamma(\frac{k}{2}) } t^{\frac{k}{2}-1} e^{-\frac{t}{2}}.
$$
with respect to the Lebesgue measure on $\mathbb{R}_+$. Changing variables shows that
$$
\chi^2(k)= 2 \gamma_{\frac{k}{2}}.
$$
Therefore the moments of $\chi^2(k)$ are given by
$$
\E \left(  \chi_k^s \right) = \frac{2^s \Gamma(\frac{k}{2}+s)}{\Gamma(\frac{k}{2})} .
$$
This allows one to check that the moments of $P_{H_N}(0)$ match those of a product of independent such variables. Since the chi-squared distributions are characterized by their moments, we deduce 
$$
P_{H_N}(0) \stackrel{d}{=} (-1)^{\epsilon} \prod_{l_1=1}^{c_1} \chi^2(2l_1-1) \prod_{l_2=1}^{c_2} \chi^2(2l_2+1) = (-1)^{\epsilon} \prod_{k=1}^N \chi^2({2 \lfloor k/2 \rfloor+1}), 
$$
where all variables are independent.
\end{proof}

\subsection{Conditional independence for beta ensembles}\label{BetaDec}
Kostan's Theorem \ref{Kostlan} and Theorem \ref{HighPowers1} have a natural formulation in terms of conditional independence for radially symmetric beta ensembles, when the inverse temperature beta is an even integer. We assume from now on that the complex random points $(\lambda_1, \dots, \lambda_N)$ are distributed according to definition \ref{Betadef} with $\beta=2p$. 

\subsubsection{Conditional independence.}

In this Section, $I$ is a random variable with values in an discrete index space $A$. We say that a collection of random variables are conditionally independent if they are independent, conditionally on $I$. \medskip

Conditional independence is not an unusual property of random variables. De Finetti's theorem gives a general setting in which such a feature appears (see \cite{Diaconis}). One can think of $I$ as a random environment, on which the distributions of the variables depend. This is equivalent to saying that for every continuous and bounded functions,
$$
\E \left( \prod_{k=1}^N f_k(X_k) \right) 
=
 \sum_{a \in A} \PP\left(I=a\right) \prod_{k=1}^N \E \left( f_k(X_k) \ | \ I=a \right).
$$

We denote by $\bu$ the vector $(u_1, \dots, u_N) \in \bZ^N$, and by $K_{N,p}(\bu)$ the coefficient of the monomial $T_1^{u_1}\dots T_N^{u_N}$ in the $p$-th power of the Vandermonde determinant in $N$ variables $T_1, \dots, T_N$,
$$
\Delta(T_1, \dots, T_N)^p = \sum_{\bu \in Z^N} K_{N,p}(\bu) \prod_{i=1}^N T_i^{u_i}.
$$
The variable $I$ we consider takes values in $A=\bZ^N$ and depends on $N,p$ and the potential $V$. Its distribution is given by the weights
\begin{equation}\label{CoinI}
\P \left(I= \bu \right) = \frac{1}{Z_{2p,N}} K_{N,p}^2(\bu) \prod_{i=1}^N \VGamma(1+u_i)
\end{equation}
we will prove later on that this defines a probability measure on $\Z^N$.

\subsubsection{Two general results of conditional independence.}
We present a generalization of Theorems \ref{Kostlan} and \ref{HighPowers1} to the above setting. Both establish conditional independence with respect to the same latent variable $I$, described above. The first result is the analog of Kostlan's independence theorem for the radii, first established in \cite{Kost}.
\begin{theorem}[Conditional independence of the radii]\label{TossIndRadii} The squared Radii of the beta ensemble with $\beta=2p$ and radial potential $V$ are conditionally independent. Conditioning on the event $\{ I=\bu \}$, the following equality in distribution holds:
$$
\{ |\la_1|^2, \dots, |\la_N|^2 \}
\stackrel{d}{=} 
\{ X_1 , \dots, X_N \},
$$
where the variables $X_1, \dots, X_N$ have independent $\VGamma$ distributions with parameters $1+u_1, \dots, 1+u_N$.
\end{theorem}

\begin{proof}
We expand the joint density to compute the product statistics, for any measurable function $g$.
\begin{align*}
\E \left( \prod_{i=1}^N g(|\lambda_i|^2)\right) & = \frac{1}{\pi^N Z_{2p,N}} \int_{\C^N} \prod_{i<j} |z_i-z_j|^{2p} \prod_{i=1}^N g(|z_i|^2) e^{- V(|z_i|^2)} \dd m(z_i) \\
& = \frac{1}{\pi^N Z_{2p,N}} \sum_{\bu,\bv \in \Z^N} \int_{\C^N} K_{N,p}(\bu) K_{N,p}(\bv) \prod_{i=1}^N z_i^{u_i} \overline{z_i}^{v_i} g(|z_i|^2) e^{- V(|z_i|^2)} \dd m(z_i).
\end{align*}
A polar change of coordinate shows that only the terms for which $\bu =\bv$ make a non zero contribution, so that there remains
$$
 \frac{1}{Z_{2p,N}} \sum_{\bu \in \Z^N} \int_{\R_+^N} K_{N,p}(\bu)^2 \prod_{i=1}^N r_i^{u_i} g(r_i) e^{- V(r_i)} \dd r_i,
$$
where $r_i=|z_i|^2$. This in turn can be written as
$$
 \frac{1}{Z_{2p,N}} \sum_{\bu \in \Z^N} K_{N,p}(\bu)^2 \prod_{i=1}^N \int_{\R_+} g(r) r^{u_i} e^{- V(r)} \dd r =  \frac{1}{Z_{2p,N}} \sum_{\bu \in \Z^N} K_{N,p}(\bu)^2 \prod_{i=1}^N \VGamma(1+u_i) \E \left( g\left( X_i \right) \right)
$$
where $X_i \stackrel{d}{=} \gamma(V,1+u_i)$ for every $i$, as claimed. For $g=1$, this tells us that
\begin{equation}\label{ConstantZ}
Z_{2p,N} = \sum_{\bu \in \Z^N} K_{N,p}(\bu)^2 \prod_{i=1}^N \VGamma(1+u_i)
\end{equation}
so that (\ref{CoinI}) defines a probability measure. \medskip

As the above holds for any polynomial $g$, by Lemma \ref{sympol} in the Appendix this characterizes the distribution of the squared radii, and establishes conditional independence. 
\end{proof}

The second result is the analog of Rains' independence theorem for the high powers, first established in \cite{Rains1}. Note that for $p=1$ we recover the optimal bound $M \geq N$.

\begin{theorem}[Conditional independence of high powers]\label{TossIndHighPowers} For any integer $M \geq (N-1)p+1$, the image of the beta ensemble with $\beta=2p$ and radial potential $V$ exhibits conditional independence. Conditioning on the event $\{ I=\bu \}$, the following equality in distribution holds:
$$
\{ \la_1^M, \dots, \la_N^M \}
\stackrel{d}{=} 
\{ X_1^{M/2} e^{i \theta_1}, \dots, X_N^{M/2} e^{i \theta_N} \},
$$
where the variables $\theta_k, X_k$ are all independent, the angles are uniform on $[0,2\pi]$, and the variables $X_1, \dots, X_N$ have independent $\VGamma$ distributions with parameters $1+u_1, \dots, 1+u_N$..  
\end{theorem}

\begin{proof}
We expand the joint density to compute the product statistics, for any polynomial $g$.
\begin{align*}
\E \left( \prod_{i=1}^N g(\lambda_i^M)\right) & 
= \frac{1}{\pi^N Z_{2p,N}} \sum_{\bu,\bv \in \Z^N} \int_{\C^N} K_{N,p}(\bu) K_{N,p}(\bv) \prod_{i=1}^N z_i^{u_i} \overline{z_i}^{v_i} g(z_i^M) e^{- V(|z_i|^2)} \dd m(z_i).
\end{align*}
Writing $g$ as a sum of monomials, as in the proof of Theorem \ref{HighPowers2}, a polar change of coordinate shows that any term that makes a non zero contribution has relative degree $0$ in every variable. In particular, these are terms for which $u_i -v_i \equiv 0 [M]$ for every $i$. On the other hand, the Vandermonde determinant is a homogeneous polynomial of partial degree $N-1$ in each variable, and its $p$-th power therefore has degree $p(N-1)$ in each variable. It follows that, for $M> p(N-1)$, congruence is only possible if $\bu=\bv$, so that there remains
$$
 \frac{1}{\pi^N Z_{2p,N}} \sum_{\bu \in \Z^N} \int_{\C^N} K_{N,p}(\bu)^2 \prod_{i=1}^N |z_i|^{2 u_i} g(z_i^M) e^{- V(|z_i|^2)} \dd m(z_i).
$$
This in turn can be written as
$$
 \frac{1}{\pi^N Z_{2p,N}} \sum_{\bu \in \Z^N} K_{N,p}(\bu)^2 \prod_{i=1}^N \int_{\C} g(z^M) |z|^{2u_i} e^{- V(|z|^2)} \dd m(z). $$
It is straightforward to check that the expression
$$
\frac{1}{\pi \VGamma(\alpha)} |z|^{2 \alpha -2} e^{- V(|z|^2)}
$$
is the density of a complex variable with squared radius $\gamma(V,\alpha)$ and independent uniform argument $\theta$. Thus, one can write
$$
\E \left( \prod_{i=1}^N g(\lambda_i^M)\right)
=
\frac{1}{Z_{2p,N}} \sum_{\bu \in \Z^N} K_{N,p}(\bu)^2 \prod_{j=1}^N \VGamma(1+u_j) \E \left( g \left( X_j^{\frac{M}{2}} e^{i \theta_j} \right) \right)
$$
where $X_j \stackrel{d}{=} \gamma(V,1+u_j)$ for every $j$, as claimed. \medskip

As the above holds for any mixed polynomial $g(z,\overline{z})$, by Lemma \ref{sympol2} in the Appendix it characterizes the distribution of the set of powers, which establishes their conditional independence.
\end{proof}

Conditional independence with this specific latent variable $I$ appears naturally when studying some statistics of beta-ensembles for even $\beta$. The structure of intermediate powers in general, however, requires further study and proper understanding, as there seem to be no analog of Corollary \ref{ProStat} that would enable to generalize the block-decomposition that holds when $\beta=2$. \medskip

\subsubsection{Distribution of the latent variable $I$.}
It is a hard problem in general to study the distribution of the latent variable $I$. Indeed, there is no tractable formula for the coefficients $K_{N,p}$, and to generate them has exponential algorithmic complexity for large values of $N$.  \medskip

For $p=1$, the above results are coherent with those of Section \ref{PVDec}. Indeed, the latent variable $I$ then gives weight only to sequences $u_1, \dots, u_N$ such that $1+u_1, \dots, 1_+u_N$ corresponds to a permutation of $\llbracket 1,N \rrbracket$, every permutation $\sigma$ having the same weight. We can therefore give the following description of $I$,
$$
I= \left(\sigma(1)-1, \dots, \sigma(N)-1 \right) \qquad \text{where} \qquad \sigma \stackrel{d}{=} \rm{Unif}(\mathfrak{S}_N).
$$
The choice of $\sigma$, however, does not change the distribution of the set of variables, as the order of the variables is not taken into account. In this way, we recover independent variables of parameters $1,2, \dots, N$. \medskip

For $N=2$ and quadratic potential, we have the following remarkable identity.

\begin{theorem}
For $N=2$, $V(x)=x$, and any integer $p$, the distribution of $I$ is given by 
$$
I=(B, p-B) \qquad \text{where } B \stackrel{d}{=} \rm{Bin}(p,\frac{1}{2})
$$
\end{theorem}

\begin{proof}
For $N=2$ the Vandermonde determinant is $T_1-T_2$, therefore the coefficients $K_{N,p}$ are given by
$$
K_{N,p}(k,l) = (-1)^{k} \binom{p}{k} \delta_{p-k,l},
$$
which yields the following weights
$$
K_{N,p}(k,l)^2 \ k! \ l! 
= \binom{p}{k}^2 \ k! \ l! \ \delta_{p-k,l} 
= p! \binom{p}{k} \delta_{p-k,l} 
$$
We deduce the value of the constant $Z_{2p,2}$ from (\ref{ConstantZ}), and the distribution of $I$ from (\ref{CoinI}),
$$
Z_{2p,2} = 2^p p! \ ,
\qquad
\P \left( I=(k,l) \right) = 2^{-p} \binom{p}{k} \delta_{p-k,l}, 
$$
which is the claimed $\rm{Bin}(p,\frac{1}{2})$ distribution.
\end{proof}

\section*{Appendix : the problem of moments}

At some point in the article we make the assumption that all relevant variables are characterized by their multivariate moments, which is to say that they are uniquely determined by the statistics
$$
\E \left( P(X_1, \dots, X_N) \right), \qquad P \in \C[T_1, \dots, T_N]
$$
when $(X_1, \dots, X_N)$ is an $N$-tuple of real random variables, and by the statistics
$$
\E \left( P(Z_1, \overline{Z_1}, \dots, Z_N, \overline{Z_N}) \right), \qquad P \in \C[T_1,S_1, \dots, T_{N},S_N]
$$
when $(Z_1, \dots, Z_N)$ is an $N$-tuple of complex random variables -- for clarity, we call such statistics mixed moments. \medskip

Our purpose here is not to inquire about the weakest possible assumptions under which the above is true. The following fact will be sufficient: it is known that a random variable $\mathbf{X}$ on $\R^N$ is characterized by its multivariate moments if it has exponential moments for any $\epsilon >0 $,
$$
\E  \left( e^{\epsilon \| \mathbf{X} \|} \right) < \infty.
$$
A proof of this result can be found, for instance, in \cite{DeJeu}. The same is true for complex variables with respect to the mixed moments, as these variables can be understood as taking values in $\R^{2N}$. The linear change of variable
$$
(T,S) \mapsto (T+iS, T-iS),
$$
ensures that a polynomial in the former variables is a polynomial in the latter, and vice versa. \medskip

The philosophy of most results contained in this article is that some features of point processes are somehow hidden if we look at the joint distribution of an $N$-tuple, but appear when we characterize the distribution of their set, obtained by taking an average over all permutations of the variables. Indeed, if $\rho$ is the joint density of an $N$-tuple of variables $(Z_1, \dots, Z_N)$, the joint density of their set is
$$
\rho_{\text{set}} (z_{1}, \dots, z_{N})
=
\frac{1}{N}
\sum_{\sigma \in \mathfrak{S}_N} \rho(z_{\sigma(1)}, \dots, z_{\sigma(N)}),
$$
and the statistics of this set are given by
$$
\E_{\text{set}} \left( f( Z_{1}, \dots, Z_{N} )\right)
=
\frac{1}{N}
\sum_{\sigma \in \mathfrak{S}_N} \E \left( f(Z_{\sigma(1)}, \dots, Z_{\sigma(N)}) \right).
$$
If the distribution of an $N$-tuple of real or complex variables is characterized by its moments, as we will always assume, then the set of these same variables is characterized by its symmetric moments. That is, the distribution of the set of variables is uniquely determined by the statistics
$$
\E \left( P(X_1, \dots, X_N) \right), \qquad P \in \rm{S}_{\C}(N) := \C [T_1, \dots, T_N]^{\mathfrak{S}_N}
$$
when $(X_1, \dots, X_N)$ is an $N$-tuple of real random variables, the space of symmetric polynomials $\rm{S}_{\C}(N)$ being defined by the property
\begin{equation}
\forall \sigma \in \mathfrak{S}_{N} \qquad P(T_{\sigma(1)}, \dots, T_{\sigma(N)}) = P(T_1,\dots, T_N).
\end{equation}
The analog statistics when $(Z_1, \dots, Z_N)$ is an $N$-tuple of complex random variables are given by
$$
\E \left( P(Z_1, \overline{Z_1}, \dots, Z_N, \overline{Z_N}) \right), \qquad P \in \rm{MS}_{\C}(N)
$$
where $\rm{MS}_{\C}(N)$ is the set of mixed symmetric polynomials, defined by the following invariance property
\begin{equation}
\forall \sigma \in \mathfrak{S}_{N} \qquad P(T_{\sigma(1)}, S_{\sigma(1)}, \dots, T_{\sigma(N)}, S_{\sigma(N)}) = P(T_1, S_1, \dots, T_N, S_N).
\end{equation}
Lemmas \ref{sympol} and \ref{sympol2} establish that we can restrain our study to a specific class of symmetric polynomials that span $\rm{S}_{\C}(N)$ (respectively $\rm{MS}_{\C}(N)$) as vector spaces. We give below a proof of these two essential Lemmas.

\begin{proof}[Proof of Lemma \ref{sympol}] It is enough to see that such expressions span the monomial symmetric polynomials, defined for any $N$-tuple of integers $(a_1, \dots, a_N)$ as
$$m_{(a_1,\dots ,a_N)} (T_1, \dots, T_N) = \sum_{\sigma \in \mathfrak{S}_N} \prod_{i=1}^N T_{\sigma(i)}^{a_{i}}$$
If $b_1, \dots, b_k$ are the distinct integers appearing in $(a_1, \dots, a_N)$, then for any parameter $t$ and any integer $M>N$ we expand the following element of $\rm{PS}_{\C}(N)$,
$$
Q_t(T_1, \dots, T_N):= \prod_{i=1}^N \Big( \sum_{j=1}^k t^{M^j} T_i^{b_j} \Big) = \sum_{\alpha_1+\dots+\alpha_k=N} t^{\sum \alpha_i M^i} m_{b_{\alpha}} (T_1, \dots, T_N)
$$
where $b_{\alpha}$ denotes the $N$-tuple where every $b_i$ is repeated $\alpha_i$ times. Note that $\sum \alpha_i M^i$ is an integer decomposition in base $M$ and thus characterizes the partition $\alpha$. For the sake of brevity we make use of the notation $\alpha \vdash N$ to denote partitions of $N$. \medskip

Applying this equality to a number of distinct values of $t$ equal to the number of integer partitions of~$N$, one expresses the vector $(Q_{t_{\lambda}})_{\lambda \vdash N}$ in terms of $(m_{b_{\alpha}})_{\alpha \vdash N}$ through the minor of an invertible Vandermonde determinant. The minor is itself invertible, and this gives us in turn an expression of $m_{(a_1,\dots ,a_N)} (T_1, \dots, T_N)$ as a linear combination of elements of $\rm{PS}_{\C}(N)$. \end{proof} 

The proof of the second Lemma goes along the very same lines, {\it mutatis mutandis}.

\begin{proof}[Proof of Lemma \ref{sympol2}] Instead of $N$-tuples, one considers $2N$-tuples of integers $(a_1, b_1 \dots, a_N, b_N)$ and monomial symmetric mixed polynomial defined by 
$$
m_{(a_1, b_1, \dots , a_N, b_N)} (Z_1, \dots, Z_N) = \sum_{\sigma \in \mathfrak{S}_N} \prod_{i=1}^N Z_{\sigma(i)}^{a_{i}} \overline{Z}_{\sigma(i)}^{b_{i}}
$$
If $(c_1,d_1), \dots, (c_k,d_k)$ are the distinct pairs of integers appearing in $\big((a_1, b_1), \dots, (a_N,b_N) \big)$, then for any parameter $t$ and any integer $M>N$ expanding the following element of $\rm{PMS}_{\C}(N)$,
$$
Q_t(Z_1,\overline{Z}_1 \dots, Z_N, \overline{Z}_N):= \prod_{i=1}^N \Big( \sum_{j=1}^k t^{M^j} Z_i^{c_j} \overline{Z}_i^{d_j} \Big) = \sum_{\alpha_1+\dots+\alpha_k=N} t^{\sum \alpha_i M^i} m_{(c,d)_{\alpha}} (Z_1, \dots, Z_N)
$$
where $(c,d)_{\alpha}$ denotes the $2N$-tuple where every pair $c_i,d_i$ is repeated $\alpha_i$ times. The same argument as above yields an expression of $m_{(a_1,b_1,\dots ,a_N,b_N)} (Z_1, \dots, Z_N)$ as a linear combination of elements of $\rm{PMS}_{\C}(N)$. \end{proof}

\section*{Acknowledgements}

The author would like to thank his advisor P.~Bourgade, as well as his colleagues K.~Mody and M.~Bilu for helpful discussions on these topics. \medskip

The work of the author is partially supported by his advisor's NSF grant DMS-1513587.

\end{document}